%% file: VEM-supg.tex
\documentclass[a4paper,twoside,10pt]{article}

\usepackage[english]{babel}
\usepackage[utf8]{inputenc}
\usepackage{lmodern,color}
\usepackage[T1]{fontenc}
\usepackage{indentfirst}
\usepackage{amsmath,amssymb}
\usepackage[percent]{overpic}
\usepackage[a4paper,top=3cm,bottom=3cm,left=3cm,right=3cm,%
bindingoffset=5mm]{geometry}
\usepackage{lmodern}
\usepackage[T1]{fontenc}
\usepackage{lettrine}
\usepackage{booktabs}
\usepackage{subcaption}
\usepackage{authblk}
\usepackage{emp}
\usepackage{amsthm}
\usepackage{amsmath,amssymb,amsthm}
\usepackage{braket}
\usepackage{chngpage}
\usepackage{cases}
\usepackage{graphicx}
\usepackage{epstopdf}
\usepackage{tabularx}
\usepackage{multirow}
\usepackage{enumerate}
\usepackage{todonotes}

\newcommand{\numberset}{\mathbb}
\newcommand{\N}{\numberset{N}}
\newcommand{\R}{\numberset{R}}

\newcommand{\Pk}{\numberset{P}}

\renewcommand{\epsilon}{\varepsilon}
\renewcommand{\theta}{\vartheta}
\renewcommand{\rho}{\varrho}
\renewcommand{\phi}{\varphi}
\newcommand{\pp}{\boldsymbol{p}}

\newcommand{\nn}{\boldsymbol{n}}

\newcommand{\bb}{\boldsymbol{\beta}}

\newcommand{\dd}{{\rm div}}
\newcommand{\gr}{\nabla}

\def\tri{{| \! | \! |}}

\def\PN{{\Pi^{\nabla, E}_{k}}}
\def\P0{{\Pi^{0, E}_{k}}}
\def\PZ0P{{\boldsymbol{\Pi}^{0, E}_{k-1}}}
\def\PP0P{{\boldsymbol{\Pi}^{0, E}_{k}}}
\def\Stab{\mathcal{S}^E}
\def\Lstab{\mathcal{L}^E}
\def\Bstab{\mathcal{B}^E}
\def\AsupgE{\mathcal{A}_{\rm supg}^E}
\def\FsupgE{\mathcal{F}_{\rm supg}^E}
\def\Asupg{\mathcal{A}_{\rm supg}}
\def\Fsupg{\mathcal{F}_{\rm supg}}
\def\Ltstab{\widetilde{\mathcal{L}}^E}
\def\Btstab{\widetilde{\mathcal{B}}^E}
\def\AtsupgE{\widetilde{\mathcal{A}}_{\rm supg}^E}
\def\FtsupgE{\widetilde{\mathcal{F}}_{\rm supg}^E}
\def\Atsupg{\widetilde{\mathcal{A}}_{\rm supg}}
\def\Ftsupg{\widetilde{\mathcal{F}}_{\rm supg}}

\def\vint{v_{\mathcal{I}}}
\def\uint{u_{\mathcal{I}}}
\def\eint{e_{\mathcal{I}}}

\def\epi{e_{\pi}}

\def\bskew{{b^{\rm skew}}}
\def\bskewE{{b^{{\rm skew},E}}}
\def\bskewEh{{b^{{\rm skew},E}_h}}
\def\bfp{{b^E_{o, h}}}
\def\bfb{{b^E_{\partial, h}}}
\def\bbO{{\|\bb\|_{[L^{\infty}(\Omega)]^2}}}
\def\bbE{{\|\bb\|_{[L^{\infty}(E)]^2}}}
\def\cinv{{\gamma_E}}
\def\supg{{\rm supg}}
\def\errF{\eta_{\mathcal{F}}^E}
\def\errFA{\eta_{\mathcal{F},1}^E}
\def\errFB{\eta_{\mathcal{F},2}^E}
\def\erra{\eta_{a}^E}

\def\errb{\eta_{b}^E}
\def\errbp{\eta_{b, o}^{E}}
\def\errbb{\eta_{b, \partial}^{E}}
\def\errbA{\eta_{b,A}^E}
\def\errbB{\eta_{b,B}^E}

\def\ebA{\eta_{b,1}^E}
\def\ebB{\eta_{b,2}^E}
\def\ebC{\eta_{b,3}^E}
\def\ebD{\eta_{b,4}^E}
\def\ebE{\eta_{b,5}^E}
\def\ebF{\eta_{b,6}^E}
\def\ebI{\eta_{b,i}^E}

\def\errB{\eta_{\mathcal{B}}^E}
\def\errBA{\eta_{{\mathcal{B}},1}^E}
\def\errBB{\eta_{{\mathcal{B}},2}^E}
\def\errBC{\eta_{{\mathcal{B}},3}^E}
\def\errL{\eta_{\mathcal{L}}^E}
\def\errLA{\eta_{{\mathcal{L}},1}^E}
\def\errLB{\eta_{{\mathcal{L}},2}^E}

\def\reg{s}

{\left\lbrace\begin{array}{@{}l@{}}}%
{\end{array}\right.}
\theoremstyle{definition}

\theoremstyle{remark}
\newtheorem{remark}{Remark}[section]

\theoremstyle{remark}

\theoremstyle{plain}

\newtheorem{proposition}{Proposition}[section]

\newtheorem{lemma}{Lemma}[section]

\usepackage{lipsum}

\usepackage{authblk}
\usepackage{color}

\usepackage{setspace}

\author[1]{L. Beir\~ao da Veiga \thanks{lourenco.beirao@unimib.it}}

\author[1]{F. Dassi \thanks{franco.dassi@unimib.it}}

\author[2]{C. Lovadina \thanks{carlo.lovadina@unimi.it}}

\author[1]{G. Vacca \thanks{giuseppe.vacca@unimib.it}}

\affil[1]{Dipartimento di Matematica e Applicazioni,
Universit\`a degli Studi di Milano Bicocca,
Via Roberto Cozzi 55 - 20125 Milano, Italy}

\affil[2]{Dipartimento di Matematica ``F. Enriques'',
Universit\`a degli Studi di Milano,
Via via Cesare Saldini 50 - 20133 Milano, Italy}

\title{\textbf{SUPG-stabilized Virtual Elements for diffusion-convection problems: a robustness analysis}}

\date{\today}

\begin{document}

\maketitle

\begin{abstract}
The objective of this contribution is to develop a convergence analysis for SUPG-stabilized Virtual Element Methods in diffusion-convection problems that is robust also in the convection dominated regime. For the original method introduced in [Benedetto et al, CMAME 2016] we are able to show an ``almost uniform'' error bound (in the sense that the unique term that depends in an unfavorable way on the parameters is damped by a higher order mesh-size multiplicative factor). We also introduce a novel discretization of the convection term that allows us to develop error estimates that are fully robust in the convection dominated cases. We finally present some numerical result.
\end{abstract}

\section{Introduction}
\label{sec:intro}

The Virtual Element Method (VEM) was introduced in \cite{volley,autostoppisti} as a generalization of the Finite Element Method (FEM) to general polygonal and polyhedral meshes. Since its introduction, the VEM enjoyed a wide success in the Numerical Analysis and Engineering communities, both due to the encouraging results and the natural construction.

The possibility of using general polytopal meshes makes VEM suitable for diffusion problems, for instance by making much easier to adapt to complex geometry of the data (such as in basin and reservoir simulations) and to irregularities of the solution. The VEM literature on the diffusion-reaction-convection problem is indeed very wide, covering primal and mixed methods, conforming and non-conforming schemes, ranging from foundation/theoretical contributions to more applicative articles; a very short representative list being 
\cite{
BBMR-2016,
BBMR-misti-2016,
BFM-2014,
AML-2016,
CMS-2018,
BDR:2017,
BBS-2016,
BBBPS-2016,
fumagalli:2018,
fumagalli:2019,
BPV:2020, 
coulet:2020}. 
Some examples of other numerical methods for the
diffusion-reaction-convection problem that can handle polytopal meshes are 
\cite{dPDE:2015,
dPE:2015,
AGH:2013,
CGH:2014,
review:2016,
CDGH:2016}.
On the other hand, the majority of the VEM contributions assume a dominant diffusion and do not address the significant case of convection dominated problems. Indeed, as it happens for standard FEM, unless some ad-hoc modification is introduced, also the VEM is expected to suffer in convection dominated regimes, leading to very large errors unless the mesh is extremely fine.
To the best of the authors' knowledge, only in the papers \cite{berrone:2016,BBM-2018} such issue is addressed; in these articles a SUPG-stabilized Virtual Element scheme for conforming and non-conforming VEM is proposed, and analyzed both theoretically and numerically. However, the stability and convergence analysis in \cite{berrone:2016,BBM-2018} is not uniform in the diffusion/convection parameters, and therefore it cannot be used to theoretically justify the method behaviour in the convection dominated regime. Moreover, a sufficiently small mesh size $h$ is required to carry out the analysis. The main difficulty in deriving uniform error estimates for SUPG-stabilized VEM is handling a variable convection coefficient in the presence of projection operators (which are needed in the VEM construction), that partially disrupt the structure of the convection term.

The aim of the present paper is to address, in the conforming case, this challenging theoretical aspect, thus deriving convergence estimates for a slight modification of the SUPG VEM scheme of \cite{berrone:2016} 
that are robust in the involved parameters and do not require a sufficiently small $h$ condition. 
We think that, in addition to filling an important theoretical gap, having this deeper understanding is fundamental in order to develop SUPG stabilizations in more complex settings, such as fluid-dynamics problems. For instance, deriving the aforementioned proofs inspired us to propose also a novel (alternative) approach for the discretization of the convective term, in addition to the original one. For the (slightly modified) discrete convection form  introduced in \cite{berrone:2016}, we are able to show an error estimate that is ``almost uniform'' in the involved parameters, in the sense that the unique term that depends in an unfavourable way on the parameters is damped by a higher order multiplicative factor in $h$. For the novel form here proposed, we are able to show full robustness in the parameters.
Finally, for the sake of completeness we also present a few numerical results, the main objective being to make a practical comparison among some different discretization options described in the previous section.

The present paper is organized as follows. 
In Section 2 we present the continuous problem and in Section 3 we introduce some preliminaries and notation. 
Afterwards, in Section 4 we review the SUPG-stabilized Virtual Element Method under analysis, also introducing the novel convective term option.
The main contribution of this article is Section 5, where we develop the aforementioned convergence analysis. The numerical tests are shown in Section 6.


Throughout the paper, we will follow the usual notation for Sobolev spaces
and norms \cite{Adams:1975}.
Hence, for an open bounded domain $\omega$,
the norms in the spaces $W^s_p(\omega)$ and $L^p(\omega)$ are denoted by
$\|{\cdot}\|_{W^s_p(\omega)}$ and $\|{\cdot}\|_{L^p(\omega)}$ respectively.
Norm and seminorm in $H^{s}(\omega)$ are denoted respectively by
$\|{\cdot}\|_{s,\omega}$ and $|{\cdot}|_{s,\omega}$,
while $(\cdot,\cdot)_{\omega}$ and $\|\cdot\|_{\omega}$ denote the $L^2$-inner product and the $L^2$-norm (the subscript $\omega$ may be omitted when $\omega$ is the whole computational
domain $\Omega$).

\section{Continuous Problem}
\label{sec:cp}

Let $\Omega \subset \R^2$ be the computational domain and let
$\epsilon > 0$ represent the diffusive coefficient (assumed to be constant), while  
$\bb \in [L^{\infty}(\Omega)]^2$ with $\dd \bb = 0$, is the  transport advective field, and
$f\in L^2(\Omega)$ is 
the volume source term.
Then, our linear steady advection-diffusion model problem reads
\begin{equation}
\label{eq:problem-c}
\left \{
\begin{aligned}
& \text{find $u \in V$ s.t.} 
\\
& \epsilon a(u,  v) + b(u, v) = ( f,  v) \qquad \text{for all $v \in V$,}
\end{aligned}
\right.
\end{equation}
where $V=H^1_0(\Omega)$ and the bilinear forms
$a(\cdot,  \cdot) \colon V \times V \to \R$ 
and
$b(\cdot,  \cdot) \colon V \times V \to \R$ 
are
\begin{equation}
\label{eq:a-c}
a(u,  v) := \int_{\Omega} \nabla u \cdot \nabla v \, {\rm d}\Omega
\qquad \text{for all $u, v \in V$,}
\end{equation}
\begin{equation}
\label{eq:b-c}
b(u,  v) := \int_{\Omega} \bb \cdot \nabla u \, v \, {\rm d}\Omega
\qquad \text{for all $u, v \in V$.}
\end{equation}
By a direct computation, being $\dd \bb = 0$, it is easy to see that the bilinear form $b(\cdot,  \cdot)$ is skew symmetric, i.e. 
\[
b(u, v) = -b(v, u) \qquad \text{for all $u, v \in V$.}
\]
Therefore, the bilinear form $b(\cdot,  \cdot)$ is equal to its skew-symmetric part, defined as:
\begin{equation}
\label{eq:bskew-c}
\bskew(u,  v) := \frac{1}{2} \bigl (b(u, v) - b(v, u) \bigr)
\qquad \text{for all $u, v \in V$.}
\end{equation}
However, at the discrete level $b(\cdot,  \cdot)$ and $\bskew(\cdot,\cdot)$ will lead to different bilinear forms, in general.

It is well known that discretizing problem \eqref{eq:problem-c} leads to instabilities when the convective
term $\bbO$ is dominant with respect to the diffusive term $\epsilon$ (see for instance \cite{quarteroni:2008}).
In such situations a stabilized form of the problem is required in order to prevent spurious oscillations that can completely spoil the numerical solution. 
In the following sections we propose a virtual elements version of the classical Streamline Upwind Petrov Galerkin (SUPG) approach \cite{hughes:1982,franca:1992}.
From now on, we assume that the material parameters are scaled so that it holds:

\smallskip\noindent
\textbf{(A0) Problem scaling.} $\bbO=1$.

\smallskip
We finally remark that the proposed approach can be trivially extended to more general situations such as reaction-convection-diffusion problems, non-constant diffusive coefficients and different boundary conditions. We here prefer to focus on the fundamental difficulties of the problem rather than deploying a ``smokescreen'' of additional minor technicalities. Moreover, also the analysis of the three dimensional case could be developed with very similar arguments.

\section{Definitions and preliminaries}
\label{sec:notations}

\subsection{SUPG stabilizing form}
\label{sub:supg form}

From now on, we will denote with $E$ a general polygon,  $e$ will denote a general edge of $E$, moreover $|E|$ and $h_E$ will denote the area  and the diameter of $E$ respectively, whereas 
$\nn^E$ will denote the unit outward normal vector to $\partial E$.
Let $\set{\Omega_h}_h$ be a sequence of decompositions of $\Omega$ into general polygons $E$, 
where $h = \sup_{E \in \Omega_h} h_E$.
We suppose that $\set{\Omega_h}_h$ fulfils the following assumption:\\
\textbf{(A1) Mesh assumption.}
There exists a positive constant $\rho$ such that for any $E \in \set{\Omega_h}_h$   
\begin{itemize}
\item $E$ is star-shaped with respect to a ball $B_E$ of radius $ \geq\, \rho \, h_E$;
\item any edge $e$ of $E$ has length  $ \geq\, \rho \, h_E$.
\end{itemize}
We remark that the hypotheses above, though not too restrictive in many practical cases, could possibly be further relaxed, combining the present analysis with the studies in~\cite{BLR:2017,brenner-sung:2018,chen-huang:2018,BV:2020}.

We now briefly review the construction of the SUPG stabilization \cite{hughes:1982,franca:1992} for the advection-dominated problem \eqref{eq:problem-c}. 
First of all, we decompose the bilinear forms
$a(\cdot, \cdot)$ and $\bskew(\cdot, \cdot)$ into local contributions, by defining
\[
a(u,  v) =: \sum_{E \in \Omega_h} a^E(u,  v) \,,
\qquad
\bskew(u,  v) =: \sum_{E \in \Omega_h} \bskewE(u,  v)
\,.
\]
Let us introduce the bilinear form $\AtsupgE(\cdot, \cdot)$, defined for all sufficiently regular functions by:
\begin{equation}
\label{eq:AtsupgE}
\AtsupgE(u, v) := \epsilon \, a^E(u, v)  + \bskewE(u, v) + \Btstab(u, v) + \Ltstab(u, v),
\end{equation}
where
\begin{align}
\label{eq:BtstabE}
\Btstab(u, v) &:= \tau_E \int_E \bb \cdot \nabla u \, (\bb \cdot \nabla v) \, {\rm d}E 
\\
\label{eq:LtstabE}
\Ltstab(u, v) &:= \tau_E \int_E -\epsilon \, \Delta u \,  (\bb \cdot \nabla v) \, {\rm d}E \,,
\end{align}
and the SUPG parameter $\tau_E > 0$ has to be chosen.
The corresponding stabilized right-hand side $\FtsupgE(\cdot)$ is defined by
\begin{equation}
\label{eq:FtsupgE}
\FtsupgE(v) := \int_E f \, v \, {\rm d}E + \tau_E \int_E f \, \bb \cdot  \nabla v \, {\rm d}E \,.
\end{equation}
The global approximated bilinear form $\Atsupg(\cdot,  \cdot)$ and the global right-hand side 
are defined by simply summing the local contributions:
\begin{align}
\label{eq:Atsupg}
\Atsupg(u, v)  &:= \sum_{E \in \Omega_h} \AtsupgE(u,  v) 
\\
\label{eq:Ftsupg}
\Ftsupg(v) &:= \sum_{E \in \Omega_h} \FtsupgE(v)  \,.
\end{align}
Since the exact solution $u$ of equation \eqref{eq:problem-c} satisfies 
$-\epsilon \, \Delta u + \bb \cdot \nabla u = f \in L^2(\Omega)$,
then $\Atsupg(u, v)$ is well defined for all $v \in V$ and $u$ solves the stabilized problem
\begin{equation}
\label{eq:supg}
\left \{
\begin{aligned}
& \text{find $u \in V$ s.t.} 
\\
& \Atsupg(u, \, v) = \Ftsupg(v) \qquad \text{for all $v \in V$.}
\end{aligned}
\right.
\end{equation}

The aim of the following sections is to derive a VEM discretization of the stabilized problem \eqref{eq:supg}.
In the following the symbol $\lesssim$ will denote a bound up to a generic positive constant, independent of the mesh size $h$, of the SUPG parameter $\tau_E$, of the diffusive coefficient $\epsilon$ and of the transport advective field $\bb$, 
but which may depend on  $\Omega$, on the ``polynomial'' order of the method $k$ and on the regularity constant appearing in the mesh assumption \textbf{(A1)}.

\subsection{Projections and polynomial approximation properties}
\label{sub:proj}

In the present subsection we introduce some basic tools and notations useful in the construction and the theoretical analysis of Virtual Element Methods.

Using standard VEM notations, for $n \in \N$, $m\in \N$ and $p= 1, \dots, \infty$, and for any $E \in \Omega_h$,  
let us introduce the spaces:
\begin{itemize}
\item $\Pk_n(\omega)$: the set of polynomials on $\omega$ of degree $\leq n$  (with $\Pk_{-1}(\omega)=\{ 0 \}$),
\item $\Pk_n(\Omega_h) := \{q \in L^2(\Omega) \quad \text{s.t} \quad q|_E \in  \Pk_n(E) \quad \text{for all $E \in \Omega_h$}\}$,
\item $W^m_p(\Omega_h) := \{v \in L^2(\Omega) \quad \text{s.t} \quad v|_E \in  W^m_p(E) \quad \text{for all $E \in \Omega_h$}\}$ equipped with the broken norm and seminorm
\[
\begin{aligned}
&\|v\|^p_{W^m_p(\Omega_h)} := \sum_{E \in \Omega_h} \|v\|^p_{W^m_p(E)}\,,
\qquad 
&|v|^p_{W^m_p(\Omega_h)} := \sum_{E \in \Omega_h} |v|^p_{W^m_p(E)}\,,
\qquad & \text{if $1 \leq p < \infty$,}
\\
&\|v\|_{W^m_p(\Omega_h)} := \max_{E \in \Omega_h} \|v\|_{W^m_p(E)}\,,
\qquad 
&|v|_{W^m_p(\Omega_h)} := \max_{E \in \Omega_h} |v|_{W^m_p(E)}\,,
\qquad & \text{if $p = \infty$,}
\end{aligned}
\]
\end{itemize}
and the following polynomial projections:
\begin{itemize}
\item the $\boldsymbol{L^2}$\textbf{-projection} $\Pi_n^{0, E} \colon L^2(E) \to \Pk_n(E)$, given by
\begin{equation}
\label{eq:P0_k^E}
\int_Eq_n (v - \, {\Pi}_{n}^{0, E}  v) \, {\rm d} E = 0 \qquad  \text{for all $v \in L^2(E)$  and $q_n \in \Pk_n(E)$,} 
\end{equation} 
with obvious extension for vector functions $\boldsymbol{\Pi}^{0, E}_{n} \colon [L^2(E)]^2 \to [\Pk_n(E)]^2$;

\item the $\boldsymbol{H^1}$\textbf{-seminorm projection} ${\Pi}_{n}^{\nabla,E} \colon H^1(E) \to \Pk_n(E)$, defined by 
\begin{equation}
\label{eq:Pn_k^E}
\left\{
\begin{aligned}
& \int_E \gr  \,q_n \cdot \gr ( v - \, {\Pi}_{n}^{\nabla,E}   v)\, {\rm d} E = 0 \quad  \text{for all $v \in H^1(E)$ and  $q_n \in \Pk_n(E)$,} \\
& \int_{\partial E}(v - \,  {\Pi}_{n}^{\nabla, E}  v) \, {\rm d}s= 0 \, ,
\end{aligned}
\right.
\end{equation}
\end{itemize}
with global counterparts 
$\Pi_n^{0} \colon L^2(\Omega) \to \Pk_n(\Omega_h)$ and
${\Pi}_{n}^{\nabla} \colon H^1(\Omega_h) \to \Pk_n(\Omega_h)$
defined by
\begin{equation}
\label{eq:proj-global}
(\Pi_n^{0} v)|_E = \Pi_n^{0,E} v \,,
\qquad
(\Pi_n^{\nabla} v)|_E = \Pi_n^{\nabla,E} v \,,
\qquad
\text{for all $E \in \Omega_h$.}
\end{equation}

We finally mention two classical results for polynomials on star-shaped domains (see for instance \cite{brenner-scott:book}).

\begin{lemma}[Bramble-Hilbert]
\label{lm:bramble}
Under the assumption \textbf{(A1)}, for any $E \in \Omega_h$ and for any  smooth enough function $\phi$ defined on $E$, it holds 
\[
\begin{aligned}
&\|\phi - \Pi^{0,E}_n \phi\|_{W^m_p(E)} \lesssim h_E^{s-m} |\phi|_{W^s_p(E)} 
\qquad & \text{$s,m \in \N$, $m \leq s \leq n+1$, $p=1, \dots, \infty$,}
\\
&\|\phi - \Pi^{\nabla,E}_n \phi\|_{m,E} \lesssim h_E^{s-m} |\phi|_{s,E} 
\qquad & \text{$s,m \in \N$, $m \leq s \leq n+1$, $s \geq 1$,}
\\
&\|\nabla \phi - \boldsymbol{\Pi}^{0,E}_{n} \nabla \phi\|_{m,E} \lesssim h_E^{s-1-m} |\phi|_{s,E} 
\qquad & \text{$s,m \in \N$, $m \leq s \leq n+2$, $s \geq 1$.}
\end{aligned}
\]
\end{lemma}

\begin{lemma}[Inverse estimate]
\label{lm:inverse}
Let, for any $E \in \Omega_h$, $\cinv$ denote the smallest positive constant such that 
for any $\pp_n \in [\Pk_n(E)]^2$,  it holds
\[
\|\dd \pp_n\|^2_{0, E} \leq \cinv h_E^{-2} \|\pp_n\|^2_{0,E} \,.
\]
Then, under assumption {\bf (A1)}, there exists $\gamma \in {\mathbb R}^+$ such that $\cinv \le \gamma$ for all $E \in \{ \Omega_h \}_h$.
\end{lemma}

\section{Virtual Element Discretization}
\label{sub:VEM}

\subsection{Virtual Element spaces}
\label{sub:spaces}

Let $k \geq 1$ be the ``polynomial'' order of the method. 
For any $E \in \Omega_h$  we consider the local ``enhanced'' virtual element space \cite{projectors} given by
\begin{equation}
\label{eq:vem space}
\begin{aligned}
V_h(E) = 
\bigl\{
v_h \in H^1(E) \cap & C^0(\partial E) \quad \text{s.t.} \quad 
 v_h|_e  \in \Pk_k(e) \quad \text{for all $e \in \partial E$,} 
\bigr .
\\
\bigl .
& \Delta v_h \in \Pk_k(E) \,, \quad
(v - \PN v, \, \widehat{p}_k ) = 0 \quad 
\text{for all $\widehat{p}_k \in \Pk_k(E) / \Pk_{k-2}(E)$}
\bigr \} \,.
\end{aligned}
\end{equation}
We here summarize the main properties of the space $V_h(E)$
(we refer to \cite{projectors} for a deeper analysis).

\begin{itemize}
\item [\textbf{(P1)}] \textbf{Polynomial inclusion:} $\Pk_k(E) \subseteq V_h(E)$;


\item [\textbf{(P2)}] \textbf{Degrees of freedom:}
the following linear operators $\mathbf{D_V}$ constitute a set of DoFs for $V_h(E)$:
\begin{itemize}
\item[$\mathbf{D_V1}$] the values of $v_h$ at the vertexes of the polygon $E$,
\item[$\mathbf{D_V2}$] the values of $v_h$ at $k-1$ distinct points of every edge $e \in \partial E$,
\item[$\mathbf{D_V3}$] the moments of $v_h$ against a polynomial basis  $\{m_i\}_i$ of $\Pk_{k-2}(E)$  s.t. $\|m_i\|_{L^{\infty}(E)} = 1$:
$$
\frac{1}{|E|}\int_E v_h \, m_{i} \, {\rm d}E  \,;
$$
\end{itemize}
\item [\textbf{(P3)}] \textbf{Polynomial projections:}
the DoFs $\mathbf{D_V}$ allow us to compute the following linear operators:
\[
\PN \colon V_h(E) \to \Pk_k(E), \qquad
\P0 \colon V_h(E) \to \Pk_{k}(E), \qquad
\PP0P \colon \nabla V_h(E) \to [\Pk_{k}(E)]^2 \,.
\]
\end{itemize}
The global virtual element space is obtained by gluing such local spaces, i.e.
\begin{equation}
\label{eq:vem global space}
V_h(\Omega_h) = \{v_h \in V \quad \text{s.t.} \quad v_h|_E \in V_h(E) \quad \text{for all $E \in \Omega_h$} \} 
\end{equation}
with the associated set of degrees of freedom.

We finally recall from \cite{cangiani:2017,brenner-sung:2018} the optimal approximation property for the space $V_h(\Omega_h)$.

\begin{lemma}[Approximation using virtual element functions]
\label{lm:interpolation}
Under the assumption \textbf{(A1)} for any $v \in V \cap H^{s+1}(\Omega_h)$ there exists $\vint \in V_h(\Omega_h)$ such that for all $E \in \Omega_h$ it holds
\[
\|v - \vint\|_{0,E} + h_E \|\nabla (v - \vint)\|_{0,E} \lesssim h_E^{s+1} |v|_{s+1,E} \,. 
\]
where $0 < s \le k$.
\end{lemma}

\subsection{Virtual Element Forms}
\label{sub:forms}

The next step in the construction of our method is to define a discrete versions of the stabilized SUPG form $\Atsupg(\cdot, \cdot)$ in \eqref{eq:AtsupgE}.
It is clear that for an arbitrary pair $(u_h,  v_h) \in V_h(E) \times V_h(E)$, the quantity 
$\AtsupgE(u_h,  v_h)$ is not computable since $u_h$ and $v_h$ are not known in closed form.
Therefore, following the usual procedure in the VEM setting, we need to construct a computable discrete bilinear form.
In the following, in accordance with definition \eqref{eq:AtsupgE}, we define a discrete counterpart of each brick composing $\AtsupgE$.

Exploiting property \textbf{(P3)}, let 
$a_h^E(\cdot,  \cdot) \colon V_h(E) \times V_h(E) \to \R$ 
be a computable approximation of the continuous form $a^E(\cdot, \cdot)$, 
defined  for all $u_h$, $v_h \in V_h(E)$ by 
 \begin{equation}
\label{eq:aEh}
a_h^E(u_h,  v_h) :=
\int_E \PZ0P \nabla u_h \cdot \PZ0P \nabla v_h \, {\rm d}E + 
\Stab((I - \PN) u_h, \, (I - \PN) v_h).
\end{equation}
Here, the stabilizing bilinear form $\Stab(\cdot, \cdot) \colon V_h(E) \times V_h(E) \to \R$  satisfies
\begin{equation}
\label{eq:sEh}
\alpha_*|v_h|_{1,E}^2 \leq \Stab(v_h, v_h) \leq \alpha^* |v_h|_{1,E}^2
\qquad \text{for all $v_h \in {\rm Ker}(\PN)$}
\end{equation}
for two positive uniform constants $\alpha_*$ and $\alpha^*$.
The condition above essentially requires the stabilizing term $\Stab(v_h, v_h)$ to scale as
$|v_h|_{1,E}^2$. 
For instance, the standard choices for the stabilization are the  \texttt{dofi-dofi} stabilization \cite{volley} and the \texttt{D-recipe} stabilization
introduced in \cite{BDR:2017}.

Concerning the approximation of the convective term $b^E(\cdot, \cdot)$, we here propose two possible choices: 
recalling property \textbf{(P3)}, let us define for all $u_h, v_h \in V_h(E)$ the following computable bilinear forms
\begin{align}
\label{eq:bfp}
\bfp(u_h, v_h) &:= \int_E \bb \cdot \PP0P \nabla u_h  \, \P0 v_h \, {\rm d}E \,,
\\
\label{eq:bfb}
\bfb(u_h, v_h) &:= \int_E \bb \cdot \nabla \P0 u_h  \, \P0 v_h \, {\rm d}E + 
\int_{\partial E} (\bb \cdot \nn^E) (I - \P0) u_h  \, v_h \, {\rm d}s \,.
\end{align}
While the form \eqref{eq:bfp} follows a more standard ``approximation by projection'' VEM approach (see for instance \cite{berrone:2016}), the novel form \eqref{eq:bfb} is amenable to the development of an improved theoretical result.
In the following $b_h^E(\cdot, \cdot) \colon V_h(E) \times V_h(E) \to \R$ will denote indifferently one of the aforementioned forms and, in accordance with \eqref{eq:bskew-c}, for all $u_h, v_h \in V_h(E)$ we define 
\begin{equation}
\label{eq:bskewEh}
\bskewEh(u_h, v_h) := \frac{1}{2} \bigl (b_h^E(u_h, v_h) - b_h^E(v_h, u_h) \bigr) \,.
\end{equation}
Exploiting again property \textbf{(P3)}, the stabilized forms $\Btstab(\cdot, \cdot)$ in \eqref{eq:BtstabE} and $\Ltstab(\cdot, \cdot)$ in \eqref{eq:LtstabE} are discretized as follows
\begin{align}
\label{eq:bstab}
\Bstab(u_h, v_h) &:= \tau_E \int_E \bb \cdot \PZ0P \nabla u_h \, \bb \cdot \PZ0P \nabla v_h \, {\rm d}E +
\tau_E \beta_E^2 \Stab((I - \PN)u_h, (I - \PN)v_h)
\\
\label{eq:lstab}
\Lstab(u_h, v_h) &:= \tau_E \int_E -\epsilon \, \dd \PZ0P \nabla u_h \,  \bb \cdot \PZ0P \nabla v_h \, {\rm d}E
\end{align}
where $\beta_E := \bbE$ and the parameter $\tau_E > 0$ has to be chosen.

In accordance with \eqref{eq:AtsupgE}, the  VEM stabilized form $\AsupgE(\cdot, \cdot) \colon V_h(E) \times V_h(E) \to \R$  is defined by
\begin{equation}
\label{eq:AsupgE}
\AsupgE(u_h, v_h) := \epsilon \, a_h(u_h, v_h)  + \bskewEh(u_h, v_h) + \Bstab(u_h, v_h) + \Lstab(u_h, v_h)
\end{equation}
for all $u_h, v_h \in V_h(E)$.

The global approximated bilinear form $\Asupg(\cdot,  \cdot) \colon V_h(\Omega_h) \times V_h(\Omega_h) \to \R$ is thus defined by summing the local contributions, i.e.
\begin{equation}
\label{eq:Asupg}
\Asupg(u, v)  := \sum_{E \in \Omega_h} \AsupgE(u_h,  v_h) 
\qquad \text{for all $u_h, v_h \in V_h(\Omega_h)$.}
\end{equation}
The corresponding computable VEM version of the SUPG right-hand side in \eqref{eq:FtsupgE} reads as 
\begin{equation}
\label{eq:FsupgE}
\FsupgE(v_h) := \int_E f \, \P0 v_h \, {\rm d}E + \tau_E \int_E f \, \bb \cdot \PZ0P \nabla v_h \, {\rm d}E 
\end{equation}
and its global counterpart is 
\begin{equation}
\label{eq:Fsupg}
\Fsupg(v_h) := \sum_{E \in \Omega_h} \FsupgE(v_h) 
\qquad \text{for all $v_h \in V_h(\Omega_h)$.}
\end{equation}

\subsection{Virtual Element SUPG problem}
\label{sub:problem}

Referring to the discrete space \eqref{eq:vem global space}, the discrete bilinear form \eqref{eq:Asupg}  and the approximated right-hand side \eqref{eq:Fsupg}, 
the virtual element SUPG approximation of the advection-dominated diffusion equation \eqref{eq:problem-c} is 
\begin{equation}
\label{eq:supg-vem}
\left \{
\begin{aligned}
& \text{find $u_h \in V_h(\Omega_h)$ s.t.} 
\\
& \Asupg(u_h, \, v_h) = \Fsupg(v) \qquad \text{for all $v_h \in V_h(\Omega_h)$.}
\end{aligned}
\right.
\end{equation}

\section{Theoretical analysis}
\label{sec:theory}
In this section we analyze the stabilization method defined in \eqref{eq:supg-vem}.
In  particular, we assess the stability property of problem \eqref{eq:supg-vem} and we provide the convergence error estimate for the discrete solution obtained with both discrete convective forms defined in \eqref{eq:bfp} and \eqref{eq:bfb}. 
All estimates clearly display the dependence on the mesh size $h$, on the parameter $\tau_E$ and the problem data $\epsilon$ and $\bb$.

\subsection{Stability}
\label{sub:stability}

Let us start with the stability analysis for the proposed VEM SUPG method.
First of all we define the VEM SUPG norm
\begin{equation}
\label{eq:loc_nom_def}
\|v_h\|^2_{\supg, E} := 
\epsilon \, \|\nabla v_h\|^2_{0,E} + 
\tau_E \, \|\bb \cdot \PZ0P \nabla v_h\|^2_{0,E} + 
\tau_E \, \beta_E^2 \, \|\nabla(I - \PN ) v_h\|^2_{0,E} 
\end{equation}
with global counterpart
\begin{equation}
\label{eq:glb_norm_def}
\|v_h\|^2_{\supg} := \sum_{E \in \Omega_h} \|v_h\|^2_{\supg, E} \,.
\end{equation}

\begin{proposition}[Coercivity]
\label{prp:coerciviity}
Under the assumption \textbf{(A1)}  if the parameters $\tau_E$ satisfy
\begin{equation}
\label{eq:tauEstab}
\tau_E \leq \frac{h_E^2}{\epsilon \cinv} \qquad \forall E \in \Omega_h
\end{equation}
where $\cinv$ is the constant appearing in the inverse estimate of Lemma \ref{lm:inverse}, the bilinear form $\AsupgE(\cdot, \cdot)$  satisfies for all $v_h \in V_h(E)$ the coerciveness inequality
\[
\|v_h\|^2_{\supg, E} \lesssim  \AsupgE(v_h, v_h) \,.
\]
\end{proposition}

\begin{proof}
We simply consider all the terms in the sum \eqref{eq:AsupgE}.
For the first three terms by definitions \eqref{eq:aEh}, \eqref{eq:bskewEh} and \eqref{eq:bstab} 
and stability estimate \eqref{eq:sEh} we get
\begin{equation}
\label{eq:stab1}
\begin{aligned}
\epsilon \, a_h^E(v_h, v_h) & \geq \epsilon \, \|\PZ0P \nabla v_h\|^2_{0, E} + 
 \epsilon \, \alpha_* \, \|\nabla v_h - \nabla \PN v_h\|^2_{0, E}
\\
\bskewEh(v_h, v_h) &= 0
\\
\Bstab(v_h, v_h) & \geq \tau_E \, \|\bb \cdot \PZ0P \nabla v_h\|^2_{0, E} +
\tau_E \, \beta_E^2 \,  \alpha_* \, \|\nabla v_h - \nabla \PN v_h\|^2_{0, E}
\end{aligned}
\end{equation}
whereas for the last term we infer
\begin{equation}
\label{eq:stab2}
\begin{aligned}
\Lstab(v_h, v_h) &= \tau_E \int_E -\epsilon \, \dd \PZ0P \nabla v_h \,  \bb \cdot \PZ0P \nabla v_h \, {\rm d}E
\\
& \geq -\frac{1}{2}\tau_E \, \epsilon^2 \|\dd \PZ0P \nabla v_h\|^2_{0,E} 
-\frac{1}{2}\tau_E \, \|\bb \cdot \PZ0P \nabla v_h\|^2_{0, E}
\quad & \text{(Cauchy-Schwarz)}
\\
& \geq -\frac{1}{2}\tau_E \, \gamma_E \, h_E^{-2} \epsilon^2 \|\PZ0P \nabla v_h\|^2_{0,E} 
-\frac{1}{2}\tau_E \, \|\bb \cdot \PZ0P \nabla v_h\|^2_{0, E}
\quad & \text{(Lemma \ref{lm:inverse})}
\\
& \geq -\frac{1}{2} \epsilon \, \|\PZ0P \nabla v_h\|^2_{0,E} 
-\frac{1}{2}\tau_E \, \|\bb \cdot \PZ0P \nabla v_h\|^2_{0, E}
\quad & \text{(bound \eqref{eq:tauEstab})}
\end{aligned}
\end{equation}
Moreover, by definition of $L^2$-orthogonal projection \eqref{eq:P0_k^E}, being $\nabla \PN v_h \in [\Pk_{k-1}(E)]^2$, it holds
\begin{equation}
\label{eq:PN-PZ0P}
\|\nabla v_h - \nabla \PN v_h\|^2_{0, E} \geq 
\|\nabla v_h - \PZ0P \nabla v_h\|^2_{0, E} \,.
\end{equation}
Collecting the previous bound, \eqref{eq:stab1} and \eqref{eq:stab2} we obtain
\[
\begin{aligned}
\AsupgE(v_h, v_h) & \geq  
\frac{1}{2}  \epsilon  \, \|\PZ0P \nabla v_h\|^2_{0,E} +
\frac{1}{2}  \tau_E \,  \|\bb \cdot \PZ0P \nabla v_h\|^2_{0, E}   +
\\
& \qquad  
+  \alpha_* \, \epsilon \, \|\nabla v_h - \PZ0P \nabla v_h\|^2_{0, E}  + 
\alpha_* \, \tau_E \,  \beta_E^2 \, \|\nabla (I- \PN)  v_h\|^2_{0, E}
\\
& \geq  \min \left\{ \frac{1}{2}, \alpha_*\right\} \|v_h\|^2_{\supg, E} \,.
\end{aligned}
\]
\end{proof}

\begin{remark}
\label{rm:norm}
Notice that the norm $\|\cdot\|_{\supg, E}$ is slightly different from the usual norm introduced in standard SUPG
theory \cite{hughes:1982,franca:1992}, i.e.
\[
\|v_h\|^2_{\widetilde{\rm supg}, E} := 
\epsilon \, \|\nabla v_h\|^2_{0,E} + 
\tau_E \, \|\bb \cdot \nabla v_h\|^2_{0,E} \,.
\]
However we observe that the ``classical norm''  $\|\cdot\|^2_{\widetilde{\rm supg}, E}$ is controlled by the  ``VEM norm'' $\|\cdot\|_{\supg, E}$. Indeed, recalling \eqref{eq:PN-PZ0P}, for any $v_h \in H^1(E)$ it holds
\[
\begin{aligned}
\|\bb \cdot \nabla v_h\|^2_{0,E}  
& \leq 
2 \|\bb \cdot \PZ0P\nabla v_h\|^2_{0,E} + 
2 \beta_E^2 \| (I - \PZ0P) \nabla v_h\|^2_{0,E} 
\\
& \leq 
2 \|\bb \cdot \PZ0P\nabla v_h\|^2_{0,E} + 
2 \beta_E^2 \| \nabla(I - \P0) v_h\|^2_{0,E}
\,.
\end{aligned}
\]

\end{remark}

\subsection{Error estimates}
\label{sub:error}

The aim of the present section is to derive the rate of convergence for the proposed SUPG virtual element scheme \eqref{eq:supg-vem} in terms of the mesh size $h$, the SUPG parameter $\tau_E$, the diffusive coefficient $\epsilon$ and transport advective field $\bb$.
The hidden constants  may depend on  $\Omega$, on $k$, on the regularity constant appearing in the mesh assumption \textbf{(A1)} and on the stability constants $\alpha_*$ and $\alpha^*$ (cf. \eqref{eq:sEh}).

Let $u \in V$ and $u_h \in V_h(\Omega_h)$ be the solutions of problem \eqref{eq:supg} and problem \eqref{eq:supg-vem}, respectively, and let us define the following error functions
\[
\eint := u - \uint \,, \qquad
\epi := u - \Pi^{\nabla}_k u \,, \qquad
e_h := u_h - \uint \,,
\]
where $\uint \in V_h(\Omega_h)$ is the interpolant function of $u$ defined in Lemma \ref{lm:interpolation}, and $\Pi^{\nabla}_k u \in \Pk_k(\Omega_h)$ is the piecewise polynomial defined in \eqref{eq:proj-global}.
We introduce the analysis with the following abstract error estimation.

\begin{proposition}
\label{prp:abstract}
Let $u \in V$ and $u_h \in V_h(\Omega_h)$ be the solutions of problem \eqref{eq:supg} and problem \eqref{eq:supg-vem}, respectively.
Then under assumption \textbf{(A1)}  if the parameters $\tau_E$ satisfy \eqref{eq:tauEstab}, it holds that
\begin{equation}
\label{eq:abstract}
\|u - u_h\|^2_{\supg} \lesssim 
\|e_{\mathcal{I}}\|^2_{\supg} + \sum_{E \in \Omega_h} \bigl(
\errF + \erra + \errb + \errB + \errL
\bigr)
\end{equation}
where
\[
\begin{aligned}
\errF &:= \FtsupgE(e_h) - \FsupgE(e_h) \,,
\\
\erra &:=  \epsilon a^E(u, e_h) - \epsilon a_h^E(\uint, e_h) \,,
\\
\errb &:= \bskewE(u, e_h) - \bskewEh(\uint, e_h) \,,
\\
\errB &:=  \Btstab(u, e_h) - \Bstab(\uint, e_h) \,,
\\
\errL &:= \Ltstab(u, e_h) - \Lstab(\uint, e_h) \,.
\end{aligned}
\]
\end{proposition}

\begin{proof}
Simple computations yield
\[
\begin{aligned}
\|e_h\|^2_{\supg} & \lesssim 
\Asupg(e_h, e_h)
=  \Asupg(u_h - \uint, e_h)
\quad & \text{(Propostion \ref{prp:coerciviity})}
\\
& \lesssim \Fsupg(e_h) - \Ftsupg(e_h) + \Atsupg(u, e_h) - \Asupg(\uint, e_h)
\quad & \text{(using \eqref{eq:supg} and \eqref{eq:supg-vem})}
\\
& \lesssim
\sum_{E \in \Omega_h} \bigl(
\errF + \erra + \errb + \errB + \errL 
\bigr)
\,\, & \text{(def. \eqref{eq:Ftsupg}, \eqref{eq:Fsupg}, \eqref{eq:Atsupg}, \eqref{eq:Asupg})}
\end{aligned}
\]
The thesis now follows by the triangular inequality.
\end{proof}
The next step in the analysis consists in estimating all the terms in the bound \eqref{eq:abstract}.
We make the following assumption:

\smallskip\noindent
\textbf{(A2) Data assumption.} The solution $u$, the advective field $\bb$ and the load $f$ in \eqref{eq:supg} satisfy:
\[
\begin{aligned}
&u \in H^{\reg+1}(\Omega_h)\,, &
&f \in H^{\reg+1}(\Omega_h)\,, &
&\bb \in [W^{\reg+1}_{\infty}(\Omega_h)]^2\,, &
\end{aligned}
\]
for some $0 < \reg \leq k$.

\smallskip\noindent
Note that in the following lemmas it is not restrictive to assume $\beta_E > 0$ since $\beta_E=0$ implies $\bb |_E=0$ and thus the corresponding terms vanish.

\begin{lemma}[Estimate of $\|\eint\|_{\supg}$]
\label{lm:epi}
Under assumptions \textbf{(A1)} and \textbf{(A2)},
the term $\|\eint\|^2_{\supg}$ can be bounded as follows (for $0 < s \le k$)
\[
\|\eint\|^2_{\supg} \lesssim 
\sum_{E \in \Omega_h} \left( \epsilon +  \tau_E \beta_E^2 \right) h^{2\reg}_E |u|_{\reg+1,E}^2 \,.
\]
\end{lemma}
\begin{proof}
Applying the definition of the norm $\|\cdot\|_{\supg,E}$, of the $L^2$-orthogonal projection $\PZ0P$ (cf. \eqref{eq:P0_k^E}), of the $H^1$-orthogonal projection $\PN$ (cf. \eqref{eq:Pn_k^E}), and the interpolation estimate of Lemma \ref{lm:interpolation}, we easily obtain
\[
\begin{aligned}
\|\eint\|^2_{\supg, E} &= 
\epsilon \|\nabla \eint\|^2_{0,E} + \tau_E \|\bb \cdot \PZ0P \nabla \eint\|^2_{0,E} +
\tau_E \beta_E^2 \|\nabla(I - \PN)\eint\|^2_{0,E}
\\
& \leq
\epsilon \|\nabla \eint\|^2_{0,E} + \tau_E  \beta_E^2 \|\nabla \eint\|^2_{0,E} +
\tau_E \beta_E^2 \|\nabla \eint\|^2_{0,E}
\lesssim \left( \epsilon +  \tau_E \beta_E^2 \right) \|\nabla \eint\|_{0,E}^2 
\\
& \lesssim 
\left( \epsilon +  \tau_E \beta_E^2 \right) h^{2\reg}_E |u|_{\reg+1,E}^2 \,.
\end{aligned}
\]
The thesis now follows by summing the local contributions.
\end{proof}

\begin{lemma}[Estimate of $\errF(e_h)$]
\label{lm:errF}
Under the assumptions \textbf{(A1)} and \textbf{(A2)},
the term $\errF$ can be bounded as follows (for $0 < s \le k$)
\[
\errF  \lesssim
\left( \lambda_E h_E^{\reg+2}  |f|_{\reg+1, E} +
\tau_E^ {1/2}  \frac{\|\bb\|_{[W^{\reg}_{\infty}(E)]^2}}{\beta_E} h_E^{\reg} \|f\|_{\reg,E}
\right) \|e_h\|_{\supg,E} 
\]
where for any $E \in \Omega_h$
\[
\lambda_E := \min \left\{ \frac{1}{\beta_E\tau^{1/2}_E}, \frac{1}{\epsilon^{1/2}} \right\} \,.
\]
\end{lemma}

\begin{proof}
Applying the definitions \eqref{eq:FtsupgE}, \eqref{eq:FsupgE} and the definition of $L^2$-orthogonal projection we obtain
\begin{equation}
\label{eq:errF}
\begin{aligned}
\errF & =
\FtsupgE(e_h) - \FsupgE(e_h)  \\
&=  
\bigl(f, \, e_h - \P0 e_h \bigr)_{0,E} + 
\tau_E \bigl(f, \, \bb \cdot (\nabla e_h -\PZ0P \nabla e_h) \bigr)_{0,E}
\qquad 
\\
& = 
\bigl((I-  \P0) f, \, (I - \P0) e_h \bigr)_{0,E} + 
\tau_E \bigl(f \bb, \,  (I -\PZ0P) \nabla e_h \bigr)_{0,E} 
\\
& = 
\bigl((I-  \P0) f, \, (I - \PN) e_h \bigr)_{0,E} + 
\tau_E \bigl((I - \PZ0P) f \bb, \,  (I -\PZ0P) \nabla e_h \bigr)_{0,E} 
\\
& =: 
\errFA + \errFB \,.
\end{aligned}
\end{equation}
Using a scaled Poincar\'e inequality  we infer
\[
\begin{aligned}
\errFA & \leq 
\|(I-  \P0) f \|_{0,E} \|(I - \PN) e_h \|_{0,E} 
\lesssim h_E\|(I-  \P0) f \|_{0,E} \|\nabla (I - \PN) e_h \|_{0,E}.
\end{aligned}
\]
Recalling the definition of the norm $\|\cdot\|_{\supg, E}$ and the stability of $\PN$ with respect to the $H^1$-seminorm, from Lemma \ref{lm:bramble} we get
\begin{equation}
\label{eq:errFA}
\errFA \lesssim \min \left \{\frac{1}{\beta_E \tau_E^{1/2}}, \frac{1}{\epsilon^{1/2}}\right\}  h_E^{\reg+2} |f|_{\reg+1,E} \|e_h\|_{\supg, E} \,.
\end{equation}
Regarding the second term $\errFB$, from \eqref{eq:PN-PZ0P} and Lemma \ref{lm:bramble} we obtain
\begin{equation}
\label{eq:errFB}
\begin{aligned}
\errFB
& \leq 
\tau_E \|(I - \PZ0P) f \bb \|_{0,E}  \|(I -\PZ0P) \nabla e_h \|_{0,E}  
\lesssim 
\tau_E^{1/2} \frac{\|(I - \PZ0P) f \bb\|_{0,E}}{\beta_E}  \|e_h\|_{\supg,E} 
\\
& \lesssim 
\tau_E^{1/2} h_E^{\reg} \frac{|f \bb|_{\reg,E}}{\beta_E}  \|e_h\|_{\supg,E} 
\lesssim
\tau_E^{1/2} \frac{\|\bb\|_{[W^{\reg}_{\infty}(E)]^2}}{\beta_E}
h_E^{\reg} \|f\|_{\reg,E}  \|e_h\|_{\supg,E} 
\,.
\end{aligned}
\end{equation}
Now the thesis follows from \eqref{eq:errF}, \eqref{eq:errFA} and \eqref{eq:errFB}.
\end{proof}

\begin{remark}
The term $\| \bb \|_{[W^{\reg}_{\infty}(E)]^2} / {\beta_E} = \| \bb / \beta_E \|_{[W^{\reg}_{\infty}(E)]^2}$ represents a locally scaled regularity term for $\bb$. Roughly speaking, it is related to the local variations of $\bb$ and not to its amplitude.
\end{remark}

\begin{lemma}[Estimate of $\erra$]
\label{lm:errA}
Under the assumptions \textbf{(A1)} and \textbf{(A2)},
the term $\erra$ can be bounded as follows (for $0 < s \le k$)
\[
\erra \lesssim
\epsilon^{1/2} h_E^{\reg} |u|_{\reg+1,E}
\|e_h\|_{\supg, E}\,.
\]
\end{lemma}

\begin{proof}
The consistency and the continuity of the form $a_h^E(\cdot, \cdot)$,
Lemma \ref{lm:bramble} and Lemma \ref{lm:interpolation}
easily imply
\[
\begin{aligned}
\erra &=
\epsilon a^E(u, e_h) -  \epsilon a_h^E(\uint, e_h)  = 
\epsilon a^E(u - \PN u, e_h)  +  \epsilon a_h^E(\PN u - \uint, e_h)
\\
& \leq
\epsilon \bigl( \|\nabla \epi\|_{0,E} + (1 + \alpha^*) \|\nabla(\PN u - \uint)\|_{0,E} \bigr) 
\|\nabla e_h\|_{0,E}
\\
& \lesssim
\epsilon \bigl( \|\nabla \eint\|_{0,E} +  \|\nabla \epi\|_{0,E} \bigr) 
\|\nabla e_h\|_{0,E}
 \lesssim
\epsilon^{1/2} 
h_E^{\reg} |u|_{\reg+1,E}
\|e_h\|_{\supg,E} \,.
\end{aligned}
\]
\end{proof}

\begin{lemma}[Estimate of $\errB$]
\label{lm:errB}
Under the assumptions \textbf{(A1)} and \textbf{(A2)},
the term $\errB$ can be bounded as follows (for $0 < s \le k$)
\[
\errB \lesssim
\tau_E^{1/2} \beta_E \frac{\|\bb\|^2_{[W^{\reg}_{\infty}(E)]^2}}{\beta^2_E}
 h_E^{\reg} \|u\|_{\reg+1,E}
\|e_h\|_{\supg,E} \,.
\]
\end{lemma}

\begin{proof}
Using the definition of $L^2$-projection, simple computations yield
\begin{equation}
\label{eq:errB}
\begin{aligned}
\errB =&
\tau_E \bigl(\bb \cdot \nabla u, \bb \cdot \nabla e_h \bigr)_{0,E} - 
\tau_E  \bigl(\bb \cdot \PZ0P \nabla \uint, \bb \cdot \PZ0P \nabla  e_h\bigr)_{0,E} 
\\
& - \tau_E \beta_E^2 \Stab((I - \PN)\uint,(I - \PN)e_h)
\\
 = & 
\tau_E  \bigl(\bb \cdot \nabla u - \bb \PZ0P \nabla \uint, \bb \cdot \PZ0P \nabla e_h\bigr)_{0,E} +  
\tau_E  \bigl(\bb \cdot  \nabla u, \bb \cdot (I -  \PZ0P) \nabla  e_h \bigr)_{0,E} 
\\
& - \tau_E \beta_E^2 \Stab((I - \PN)\uint,(I - \PN)e_h)
\\
 = & 
\tau_E  \bigl(\bb \cdot \nabla u - \bb \PZ0P \nabla \uint, \bb \cdot \PZ0P \nabla e_h\bigr)_{0,E} 
\\
& +  
\tau_E  \bigl( (I- \PZ0P) \bb \bb^{\rm T}  \nabla u,  (I -  \PZ0P) \nabla  e_h \bigr)_{0,E} 
\\
& - \tau_E \beta_E^2 \Stab((I - \PN)\uint,(I - \PN)e_h)
\\
 =:& \errBA + \errBB + \errBC \,.
\end{aligned}
\end{equation}
We analyse separately each term in the sum.
The term $\errBA$ is bounded using \eqref{eq:PN-PZ0P} and the continuity of $\PZ0P$ with respect to the $L^2$-norm,
Lemma \ref{lm:bramble} and Lemma \ref{lm:interpolation}:
\begin{equation}
\label{eq:errBA}
\begin{aligned}
\errBA &\leq 
\tau_E \|\bb \cdot \nabla u - \bb \cdot\PZ0P \nabla \uint\|_{0,E} \|\bb \cdot \PZ0P \nabla e_h\|_{0,E}
\\
&\leq \tau_E^{1/2} \beta_E \|\nabla u -  \PZ0P \nabla \uint\|_{0,E} \|e_h\|_{\supg, E}
\\
&\leq \tau_E^{1/2} \beta_E \bigl( \|(I -  \PZ0P) \nabla u\|_{0,E} +
\|\PZ0P \nabla (u - \uint)\|_{0,E}  \bigr)
\|e_h\|_{\supg, E} 
\\
&\leq \tau_E^{1/2} \beta_E \bigl( \|\nabla \epi\|_{0,E} +
\|\nabla \eint\|_{0,E}  \bigr)
\|e_h\|_{\supg, E} 
\\
&\lesssim \tau_E^{1/2} \beta_E h_E^{\reg} |u|_{\reg+1,E}
\|e_h\|_{\supg, E} \, .
\end{aligned}
\end{equation}
For the second term $\errBB$ using again \eqref{eq:PN-PZ0P} and Lemma \ref{lm:bramble} we infer
\begin{equation}
\label{eq:errBB}
\begin{aligned}
\errBB &\leq
\tau_E \|(I - \PZ0P) \bb \bb^{\rm T}  \nabla u\|_{0,E} 
\|(I -  \PZ0P) \nabla  e_h \|_{0,E}
\\
& \leq
\tau_E^{1/2} \beta_E \frac{\|(I - \PZ0P) \bb \bb^{\rm T}  \nabla u\|_{0,E} }{\beta^2_E}
\|e_h\|_{\supg, E} 
\\
&\lesssim
\tau_E^{1/2} \beta_E h_E^{\reg} \frac{|\bb \bb^{\rm T}  \nabla u|_{\reg,E} }{\beta^2_E}
\|e_h\|_{\supg, E} 
\lesssim
\tau_E^{1/2} \beta_E \frac{\|\bb\|^2_{[W^{\reg}_{\infty}(E)]^2} }{\beta^2_E}
 h_E^{\reg} \| u\|_{\reg+1,E} 
\|e_h\|_{\supg, E} 
\,.
\end{aligned}
\end{equation}
Finally for the last term in \eqref{eq:errB}, employing \eqref{eq:sEh}, the stability of the $H^1$-seminorm projection with respect to the $H^1$-seminorm, Lemma \ref{lm:bramble} and Lemma \ref{lm:interpolation} we get
\begin{equation}
\label{eq:errBC}
\begin{aligned}
\errBC & = -
\tau_E \beta_E^2 \Stab((I - \PN)\uint,(I - \PN)e_h)
\\
& \leq
\alpha^* \tau_E \beta_E^2 \|\nabla (I - \PN) \uint \|_{0,E} \|\nabla (I - \PN) e_h\|_{0,E}
\\
& \leq
\alpha^* \tau_E^{1/2} \beta_E \bigl( \|\nabla \eint\|_{0,E} + 
\|\nabla \epi\|_{0,E} \bigr) \|e_h\|_{\supg, E} 
\\
& \lesssim \tau_E^{1/2} \beta_E h_E^{\reg} |u|_{\reg+1,E} \|e_h\|_{\supg, E}
\,.
\end{aligned}
\end{equation}
The thesis now follows by collecting \eqref{eq:errBA}, \eqref{eq:errBB} and \eqref{eq:errBC} in \eqref{eq:errB}.
\end{proof}

\begin{lemma}[Estimate of $\errL$]
\label{lm:errL}
Under the assumptions \textbf{(A1)} and \textbf{(A2)},
the term $\errL$ can be bounded as follows (for $0 < s \le k$)
\[
\errL  \lesssim
\tau_E^{1/2} \epsilon 
\frac{\|\bb \|_{[W^{\max{ \{\reg-1,0 \}}}_{\infty}(E)]^2}}{\beta_E}
h_E^{\max{ \{\reg-1,0 \}} } \|u\|_{\reg+1,E} \|e_h\|_{\supg, E} \,.
\]
\end{lemma}

\begin{proof}
By definition of $L^2$-orthogonal projection we infer
\begin{equation}
\label{eq:errL}
\begin{aligned}
\errL =&
\tau_E \epsilon\bigl(\dd \PZ0P \nabla \uint, \bb \cdot \PZ0P \nabla  e_h \bigr)_{0,E} -
\tau_E \epsilon\bigl(\Delta u, \bb \cdot \nabla e_h \bigr)_{0,E}  
\\
=& \tau_E \epsilon \bigl( \dd (\PZ0P \nabla \uint - \nabla u), \bb \cdot \PZ0P \nabla  e_h\bigr)_{0,E} 
- \tau_E \epsilon \bigl(\Delta u, \bb \cdot (I - \PZ0P) \nabla e_h \bigr)_{0,E}  
\\
=& \tau_E \epsilon \bigl( \dd (\PZ0P \nabla \uint - \nabla u), \bb \cdot \PZ0P \nabla  e_h \bigr)_{0,E} 
- \tau_E \epsilon\bigl((I - \PZ0P) \Delta u \bb, (I - \PZ0P) \nabla e_h \bigr)_{0,E} 
\\
=: & \errLA + \errLB \,.
\end{aligned}
\end{equation}
The term $\errLA$, employing Lemma \ref{lm:inverse}, Lemma \ref{lm:bramble} and Lemma \ref{lm:interpolation} is estimated as follows 
\begin{equation}
\label{eq:errLA}
\begin{aligned}
\errLA &\leq
\tau_E \epsilon \| \dd (\nabla u - \PZ0P \nabla \uint)\|_{0,E} \| \bb \cdot \PZ0P \nabla  e_h\|_{0,E}
\\
&\leq
\tau_E^{1/2} \epsilon \| \dd (\nabla u - \PZ0P \nabla \uint)\|_{0,E} \|e_h\|_{\supg, E}
\\
&\leq
\tau_E^{1/2} \epsilon \left(
 \| \dd (\nabla u -  \PZ0P \nabla u)\|_{0,E} + \| \dd \PZ0P (\nabla u -   \nabla \uint)\|_{0,E}
 \right) \|e_h\|_{\supg, E}
\\
&\leq
\tau_E^{1/2} \epsilon \left(
 | (I -  \PZ0P) \nabla u|_{1,E} + h_E^{-1} \gamma_E^{1/2}\|\nabla \eint\|_{0,E}
 \right) \|e_h\|_{\supg, E} 
\\ 
&\lesssim
\tau_E^{1/2} \epsilon 
h_E^{\reg-1} |u|_{\reg+1,E}
\|e_h\|_{\supg, E} 
 \,.
\end{aligned}
\end{equation}
The second term in \eqref{eq:errL}, recalling \eqref{eq:PN-PZ0P}, can be easily bounded as follows
\begin{equation}
\label{eq:errLB}
\begin{aligned}
\errLB &\leq \tau_E \epsilon \|(I- \PZ0P) \Delta u \bb \|_{0,E} \|(I - \PZ0P) \nabla e_h\|_{0,E}
\leq 
\tau_E^{1/2} \epsilon \frac{\|(I - \PZ0P)\Delta u \bb \|_{0,E}}{\beta_E} \|e_h\|_{\supg, E} 
\\
& \lesssim
\tau_E^{1/2} \epsilon \frac{|\Delta u \bb |_{\theta,E}}{\beta_E} h_E^{\theta} \|e_h\|_{\supg, E} 
 \lesssim
\tau_E^{1/2} \epsilon \frac{\|\bb\|_{[W^{\theta}_{\infty}(E)]^2}}{\beta_E} h_E^{\theta} \|u\|_{\reg+1,E} \|e_h\|_{\supg, E} 
\,.
\end{aligned}
\end{equation}
where $\theta = \max{ \{\reg-1,0 \}}$.
Collecting \eqref{eq:errLA} and \eqref{eq:errLB} in \eqref{eq:errL} we get the thesis.
\end{proof}

The last and most challenging step in the analysis consists in estimating the term $\errb$ in \eqref{eq:abstract} for both $\bfp$ and $\bfb$ (that we denote respectively by $\errbp$ and $\errbb$), see also Remark \ref{rem:chal}.

\begin{lemma}[Estimate of $\errbp$]
\label{lm:errbp}
Let $\bfp(\cdot, \cdot)$ be the bilinear form in \eqref{eq:bfp}.
Then under assumptions \textbf{(A1)} and \textbf{(A2)},
the term $\errbp$ can be bounded as follows (for $0 < s \le k$)
\begin{equation}
\label{eq:errbpstima}
\begin{aligned}
\errbp  &\lesssim
\left( \sigma_E \frac{\|\bb\|_{[W^{\reg+1}_{\infty}(E)]^2}}{\beta_E} h_E^{\reg+1} \|u\|_{\reg+1,E}+
\frac{1}{\epsilon^{1/2}}
|\bb|_{[W^1_{\infty}(E)]^2}
h_E^{\reg+2} |u|_{\reg+1,E} \right) \|e_h\|_{\supg,E} +
\\
& \qquad  \qquad \qquad  \qquad \qquad  +
|\bb|_{[W^{\reg+1}_{\infty}(E)]^2}
h_E^{2\reg+1} |u|_{\reg+1,E} \|e_h\|_{0,E} +
\int_{\partial E} (\bb \cdot \nn^E) \eint e_h \,{\rm d}s
\end{aligned}
\end{equation}
where for any $E \in \Omega_h$
\begin{equation}
\label{eq:sigma}
\sigma_E = \min\left\{ \frac{\beta_E}{\epsilon^{1/2}}, \frac{1}{\tau_E^{1/2}} \right\} = \beta_E \lambda_E \,.
\end{equation}
\end{lemma}

\begin{proof}
By the definition of the skew symmetric forms  \eqref{eq:bskew-c} and \eqref{eq:bskewEh} we need to estimate the terms
\begin{align*}
\errbA &:=(\bb \cdot \nabla u, e_h)_E - (\bb \cdot \PP0P \nabla \uint, \P0 e_h)_E \,,
\\
\errbB &:= (\P0 \uint, \bb \cdot \PP0P \nabla e_h)_E - (u, \bb \cdot \nabla e_h)_E \,.  
\end{align*}
Using usual computations we infer
\[
\begin{aligned}
\errbA  
&= 
\bigl(\bb \cdot \nabla u, e_h \bigr)_E - 
\bigl(\bb \cdot \PP0P \nabla \uint, \P0 e_h \bigr)_E
\\
& = 
\bigl(\bb \cdot \nabla u, e_h \bigr)_E - 
\bigl(\bb \cdot \nabla \uint, \P0 e_h \bigr)_E +
\bigl(\bb \cdot (I - \PP0P) \nabla \uint, \P0 e_h \bigr)_E
\\
& =
\bigl(\bb \cdot \nabla (u - \uint), e_h \bigr)_E + 
\bigl(\bb \cdot \nabla \uint, (I- \P0) e_h \bigr)_E +
\bigl((I - \PP0P) \nabla \uint, \bb\P0 e_h \bigr)_E
\\
& =
\bigl(\bb \cdot \nabla \eint, e_h \bigr)_E + 
\bigl((I- \P0)  \bb \cdot \nabla \uint, (I- \P0) e_h \bigr)_E +
\\
& \qquad + \bigl((I - \PP0P) \nabla \uint, \bb(\P0 e_h - \Pi_0^{0,E} e_h) \bigr)_E + 
\bigl((I - \PP0P) \nabla \uint, \bb \Pi_0^{0,E} e_h \bigr)_E
\\
& =: \ebA + \ebB + \ebC + \ebD \,,
\\
\vspace{2ex}
\\
\errbB  &= 
\bigl(\P0 \uint, \bb \cdot \PP0P \nabla e_h \bigr)_E - 
\bigl(u, \bb \cdot \nabla e_h \bigr)_E
\\
& = 
\bigl( \P0 \uint - u, \bb \cdot \PP0P \nabla e_h \bigr)_E + 
\bigl(u, \bb \cdot (\PP0P - I) \nabla e_h \bigr)_E
\\
& = 
\bigl( \P0 \uint - u, \bb \cdot \PP0P \nabla e_h \bigr)_E + 
\bigl( (I - \PP0P) \bb u,   (\PP0P - I) \nabla e_h \bigr)_E
\\
& =: \ebE + \ebF \,,
\end{aligned}
\]
yielding the following expression for $\errbp$
\begin{equation}
\label{eq:errb}
2 \errbp  = \ebA + \ebB + \ebC + \ebD + \ebE + \ebF \,.
\end{equation}
We now analyse each term $\ebI$ for $i=1, \dots, 6$ in the sum above.

\noindent
$\bullet \,\, \ebA$: using an integration by parts, bound \eqref{eq:PN-PZ0P} and the definition of $\|\cdot\|_{\supg, E}$ we infer 
\begin{equation}
\label{eq:ebA}
\begin{aligned}
\ebA &= (\bb \cdot \nabla \eint, e_h)_E 
= -(\eint, \bb \cdot \nabla e_h)_E + \int_{\partial E} (\bb \cdot \nn^E) \eint e_h \, {\rm d}s 
\\
& \leq
\|\eint\|_{0,E} \|\bb \cdot \nabla e_h\|_{0,E} + \int_{\partial E} (\bb \cdot \nn^E) \eint e_h \,{\rm d}s 
\\
& \leq
\|\eint\|_{0,E} 
\left(\|\bb \cdot \PZ0P \nabla e_h\|_{0,E} + 
\beta_E \| \nabla(I - \PN) e_h\|_{0,E}
\right)
+ \int_{\partial E} (\bb \cdot \nn^E) \eint e_h \,{\rm d}s 
\\
& \lesssim
\min \left\{ \frac{\beta_E}{\epsilon^{1/2}}, \frac{1}{\tau_E^{1/2}} \right\}
\|\eint\|_{0,E} \|e_h\|_{\supg,E} + \int_{\partial E} (\bb \cdot \nn^E) \eint e_h \,{\rm d}s 
\\
& \lesssim
\sigma_E
h_E^{\reg+1} |u|_{\reg+1,E} \|e_h\|_{\supg,E} + \int_{\partial E} (\bb \cdot \nn^E) \eint e_h \,{\rm d}s
\,.
\end{aligned}
\end{equation}

\noindent
$\bullet  \,\, \ebB$: a scaled Poincar\'e inequality and the definition of $L^2$-projection imply 
\begin{equation}
\label{eq:ebB}
\begin{aligned}
\ebB & = 
\bigl((I- \P0)\bb \cdot \nabla \uint, (I- \P0) e_h \bigr)_E
\\
& = 
\bigl((I- \P0) \bb \cdot \nabla u, (I- \P0) e_h \bigr)_E - 
\bigl((I- \P0)\bb \cdot \nabla \eint, (I- \P0) e_h \bigr)_E
\\
& \leq 
\left( \|(I- \P0) \bb \cdot \nabla u\|_{0,E} + \|\bb \cdot \nabla \eint\|_{0,E}\right) 
\|(I- \P0) e_h\|_{0,E}
\\
& \leq
\left( \|(I- \P0) \bb \cdot \nabla u\|_{0,E} + \|\bb \cdot \nabla \eint\|_{0,E}\right) 
\|(I- \PN) e_h\|_{0,E}
\\
& \lesssim
\min\left\{ \frac{\beta_E}{\epsilon^{1/2}}, \frac{1}{\tau_E^{1/2}} \right\} h_E
\left( \frac{\|(I- \P0) \bb \cdot \nabla u\|_{0,E}}{\beta_E} + 
\frac{\|\bb \cdot \nabla \eint\|_{0,E}}{\beta_E}\right) 
\|e_h\|_{\supg,E} 
\\
& \lesssim
\sigma_E 
\left( \frac{|\bb \cdot \nabla u|_{\reg,E}}{\beta_E} + 
|u|_{\reg+1,E} \right) 
h_E^{\reg+1} 
\|e_h\|_{\supg,E}
\\
& \lesssim
\sigma_E 
\left( \frac{\|\bb\|_{[W^{\reg}_{\infty}(E)]^2}}{\beta_E} + 
1\right) 
h_E^{\reg+1} \|u\|_{\reg+1,E} 
\|e_h\|_{\supg,E}
\,.
\end{aligned}
\end{equation}

\noindent
$\bullet \,\, \ebC$: from the definition of $L^2$-projection, the Poincar\'e inequality and Lemma \ref{lm:bramble}, we infer 
\begin{equation}
\label{eq:ebC}
\begin{aligned}
\ebC & = 
\bigl((I - \PP0P) \nabla \uint, \bb(\P0 e_h - \Pi_0^{0,E} e_h) \bigr)_E 
\\ 
& = \bigl((I - \PP0P) \nabla \uint, (\bb - \boldsymbol{\Pi}_0^{0,E} \bb)(\P0 e_h - \Pi_0^{0,E} e_h)\bigr)_E
\\
& \leq 
\|(I - \PP0P) \nabla \uint\|_{0,E} 
\|(I - \boldsymbol{\Pi}_0^{0,E}) \bb\|_{L^{\infty}}
\|(\P0 - \Pi_0^{0,E}) e_h\|_{0,E} 
\\
& \leq
\|(I - \PZ0P) \nabla \uint\|_{0,E} 
\|(I - \boldsymbol{\Pi}_0^{0,E}) \bb\|_{L^{\infty}}
\|(I - \Pi_0^{0,E}) e_h\|_{0,E} 
\\
& \lesssim
\frac{h_E}{\epsilon^{1/2}}
\bigl( \|\nabla \eint\|_{0,E} + \|\nabla \epi\|_{0,E}  \bigr)
\|(I - \boldsymbol{\Pi}_0^{0,E}) \bb\|_{L^{\infty}}
\|e_h\|_{\supg,E} 
\\
& \lesssim
\frac{1}{\epsilon^{1/2}}
|\bb|_{[W^1_{\infty}(E)]^2}
h_E^{\reg+2} |u|_{\reg+1,E} \|e_h\|_{\supg,E}
\,.
\end{aligned}
\end{equation}

\noindent
$\bullet \,\, \ebD$: using similar computations of the previous item we obtain
\begin{equation}
\label{eq:ebD}
\begin{aligned}
\ebD & = 
\bigl((I - \PP0P) \nabla \uint, \bb \Pi_0^{0,E}  e_h \bigr)_E 
\\ 
& =
\bigr((I - \PP0P) \nabla \uint, (\bb - \PP0P \bb) \Pi_0^{0,E} e_h \bigr)_E
\\
& \leq 
\|(I - \PP0P) \nabla \uint\|_{0,E}  \|(I - \boldsymbol{\Pi}_0^{k,E}) \bb\|_{L^{\infty}}
\|\Pi_0^{0,E} e_h\|_{0,E} 
\\
& \leq
\bigl( \|\nabla \eint\|_{0,E} + \|\nabla \epi\|_{0,E}  \bigr)
\|(I - \boldsymbol{\Pi}_0^{k,E}) \bb\|_{L^{\infty}}
\|e_h\|_{0,E} 
\\
& \lesssim
|\bb|_{[W^{\reg+1}_{\infty}(E)]^2}
h_E^{2\reg+1} |u|_{\reg+1,E} \|e_h\|_{0,E}
\,.
\end{aligned}
\end{equation}

\noindent
$\bullet \,\, \ebE$: exploiting the property of $L^2$-projection and bound \eqref{eq:PN-PZ0P} we get
\begin{equation}
\label{eq:ebE}
\begin{aligned}
\ebE & = \bigl( \P0 \uint - u, \bb \cdot \PP0P \nabla e_h \bigr)_E
\\
& =
\bigl( \P0 \uint - u, \bb \cdot \PZ0P \nabla e_h \bigr)_E +
\bigl( \P0 \uint - u, \bb \cdot (\PP0P - \PZ0P) \nabla e_h \bigr)_E
\\
& \leq 
\|\P0 \uint - u\|_{0,E} 
\left(
\|\bb \cdot \PZ0P \nabla e_h\|_{0,E} + 
\beta_E \|(\PP0P - \PZ0P) \nabla e_h \|_{0,E}
\right)
\\
& \leq
\left(\|(I - \P0) u\|_{0,E} +\|\eint\|_{0,E}\right)
\left(
\|\bb \cdot \PZ0P \nabla e_h\|_{0,E} + 
\beta_E \|\nabla(I - \PN) e_h \|_{0,E}
\right)
\\
& \leq 
\min\left\{ \frac{\beta_E}{\epsilon^{1/2}}, \frac{1}{\tau_E^{1/2}} \right\}
\left(\|(I - \P0) u\|_{0,E} +\|\eint\|_{0,E}\right)
\|e_h\|_{\supg, E} 
\\
& \lesssim 
\sigma_E h_E^{\reg+1} |u|_{\reg+1,E}
\|e_h\|_{\supg, E}
\,.
\end{aligned}
\end{equation}

\noindent
$\bullet \,\, \ebF$: using similar computations of the previous item we have
\begin{equation}
\label{eq:ebF}
\begin{aligned}
\ebF & = 
\bigl( (I - \PP0P) \bb u,   (\PP0P - I) \nabla e_h \bigr)_E
\\
& \leq
\|(I - \PP0P) \bb u \|_{0,E}  \|(\PP0P - I) \nabla e_h \|_{0,E}
\leq 
\|(I - \PP0P) \bb u \|_{0,E}  \|\nabla (I - \PN)e_h \|_{0,E}
\\
& \leq
\min\left\{ \frac{\beta_E}{\epsilon^{1/2}}, \frac{1}{\tau_E^{1/2}} \right\}
\frac{\|(I - \PP0P) \bb u \|_{0,E}}{\beta_E}
\|e_h\|_{\supg, E} 
\\
& \lesssim
\sigma_E
\frac{|\bb u |_{\reg+1,E}}{\beta_E}
h_E^{\reg+1} \|e_h\|_{\supg, E}
\\
& \lesssim
\sigma_E
\frac{\|\bb\|_{[W^{\reg+1}_{\infty}(E)]^2}}{\beta_E}
h_E^{\reg+1} \|u\|_{\reg+1,E}\|e_h\|_{\supg, E}
\,.
\end{aligned}
\end{equation}
The thesis now follows gathering \eqref{eq:ebA}--\eqref{eq:ebF} in \eqref{eq:errb}.
\end{proof}

\begin{lemma}[Estimate of $\errbb$]
\label{lm:errbb}
Let $\bfb(\cdot, \cdot)$ be the bilinear form in \eqref{eq:bfb}.
Then under assumptions \textbf{(A1)} and \textbf{(A2)},
the term $\errbb$ can be bounded as follows
\[
\begin{aligned}
\errbb  \lesssim
&
\sigma_E  \frac{\|\bb\|_{[W^{\reg+1}_{\infty}(E)]^2}}{\beta_E} h_E^{\reg+1} \|u\|_{\reg+1,E} \|e_h\|_{\supg,E} +
\int_{\partial E} (\bb \cdot \nn^E) \eint e_h \,{\rm d}s
\end{aligned}
\]
where $\sigma_E$ is defined in \eqref{eq:sigma}.
\end{lemma}

\begin{proof}
Recalling definition  \eqref{eq:bfb} we need to estimate the terms
\begin{align*}
\errbA &:=(\bb \cdot \nabla u, e_h)_E - (\bb \cdot  \nabla \P0 \uint, \P0 e_h)_E 
- \int_{\partial E} (\bb \cdot \nn^E) (I - \P0)\uint \P0 e_h \, {\rm d}s \,,
\\
\errbB &:= (\P0 \uint, \bb \cdot \nabla \P0 e_h)_E - (u, \bb \cdot \nabla e_h)_E 
+ \int_{\partial E} (\bb \cdot \nn^E) (I - \P0)e_h \P0 \uint \, {\rm d}s
\,.  
\end{align*}
By integration by parts we have
\[
\begin{aligned}
\errbA  
& = 
\bigl(\bb \cdot \nabla u, (I - \P0)e_h \bigr)_E + 
\bigl(\bb \cdot  \nabla (u - \P0 \uint), \P0 e_h \bigr)_E +
\\
& \qquad -\int_{\partial E} (\bb \cdot \nn^E) (I - \P0)\uint \P0 e_h \, {\rm d}s
\\
& =
\bigl(\bb \cdot \nabla u, (I - \P0)e_h \bigr)_E - 
\bigl(u - \P0 \uint, \bb \cdot \nabla \P0 e_h \bigr)_E + 
\\
& \qquad +\int_{\partial E} (\bb \cdot \nn^E) (u - \uint) \P0 e_h \, {\rm d}s
\\
& =
\bigl((I - \P0) \bb \cdot \nabla u, (I - \P0)e_h \bigr)_E + 
\bigl(\P0 \uint - u, \bb \cdot \nabla \P0 e_h \bigr)_E + 
\\
& \qquad +\int_{\partial E} (\bb \cdot \nn^E) \eint \P0 e_h \, {\rm d}s
\\
& =: \ebA + \ebB + \ebC \,,
\end{aligned}
\]
\[
\begin{aligned}
\errbB  
& = 
\bigl(\P0 \uint - u, \bb \cdot \nabla \P0 e_h \bigr)_E - 
\bigl(u, \bb \cdot \nabla ( I - \P0) e_h \bigr)_E +
\\
& \qquad  + \int_{\partial E} (\bb \cdot \nn^E) (I - \P0)e_h \P0 \uint \, {\rm d}s
\\
& = 
\bigl(\P0 \uint - u, \bb \cdot \nabla \P0 e_h \bigr)_E + 
\bigl(\bb \cdot \nabla u,  ( I - \P0) e_h \bigr)_E +
\\
& \qquad  + \int_{\partial E} (\bb \cdot \nn^E) (I - \P0)e_h (\P0 \uint - u)\, {\rm d}s
\\
& = 
\bigl(\P0 \uint - u, \bb \cdot \nabla \P0 e_h \bigr)_E + 
\bigl(( I - \P0)\bb \cdot \nabla u,  ( I - \P0) e_h \bigr)_E +
\\
& \qquad  + \int_{\partial E} (\bb \cdot \nn^E) (I - \P0)e_h (\P0 \uint - u)\, {\rm d}s
\\
& =: \ebB + \ebA + \ebD \,,
\end{aligned}
\]
yielding the following expression for $\errbb$
\begin{equation}
\label{eq:errb2}
2 \errbb  = 2\ebA + 2\ebB + \ebC + \ebD \,.
\end{equation}
We now analyse each term $\ebI$ for $i=1, \dots, 4$ in the sum above.

\noindent
$\bullet \,\, \ebA$: using the same computations in \eqref{eq:ebB} we infer
\begin{equation}
\label{eq:ebA2}
\begin{aligned}
\ebA &= 
\bigl((I- \P0) \bb \cdot \nabla u, (I- \P0) e_h \bigr)_E 
 \lesssim
\sigma_E 
 \frac{\|\bb\|_{[W^{\reg}_{\infty}(E)]^2}}{\beta_E}
h_E^{\reg+1} \|u\|_{\reg+1,E} 
\|e_h\|_{\supg,E}
\,.
\end{aligned}
\end{equation}

\noindent
$\bullet \,\, \ebB$: exploiting the computation in \eqref{eq:ebE} we obtain
\begin{equation}
\label{eq:ebB2}
\begin{aligned}
\ebB & = \bigl( \P0 \uint - u, \bb \cdot  \nabla \P0 e_h \bigr)_E
\\
& =
\bigl( \P0 \uint - u, \bb \cdot \PZ0P \nabla e_h \bigr)_E +
\bigl( \P0 \uint - u, \bb \cdot (\nabla \P0 e_h - \PZ0P \nabla e_h)  \bigr)_E
\\
& \leq 
\|\P0 \uint - u\|_{0,E} 
\left(
\|\bb \cdot \PZ0P \nabla e_h\|_{0,E} + 
\beta_E \|\PZ0P (\nabla e_h - \nabla \P0 e_h )  \|_{0,E}
\right)
\\
& \leq
\left(\|(I - \P0) u\|_{0,E} +\|\eint\|_{0,E}\right)
\left(
\|\bb \cdot \PZ0P \nabla e_h\|_{0,E} + 
\beta_E \|\nabla(I - \PN) \nabla e_h \|_{0,E}
\right)
\\
& \leq 
\min\left\{ \frac{\beta_E}{\epsilon^{1/2}}, \frac{1}{\tau_E^{1/2}} \right\}
\left(\|(I - \P0) u\|_{0,E} +\|\eint\|_{0,E}\right)
\|e_h\|_{\supg, E} 
\\
& \lesssim 
\sigma_E h_E^{\reg+1} |u|_{\reg+1,E}
\|e_h\|_{\supg, E}
\,.
\end{aligned}
\end{equation}

\noindent
$\bullet \,\, \ebC+ \ebD$: we use a scaled trace inequality \cite{brenner-scott:book} making use of the scaled norm
$\tri v \tri_{1,E}^2 := \| v \|_{L^2(E)}^2 + h_E^2 | v |_{H^1(E)}^2$ for all $v \in H^1(E)$. We obtain
\begin{equation}
\label{eq:ebC2}
\begin{aligned}
\ebC &+ \ebD  = 
\int_{\partial E} (\bb \cdot \nn^E) \eint \P0 e_h \, {\rm d}s +
\int_{\partial E} (\bb \cdot \nn^E) (I - \P0)e_h (\P0 \uint - u)\, {\rm d}s
\\
& = 
\int_{\partial E} (\bb \cdot \nn^E) (\P0 - I)e_h (\eint + u - \P0 \uint)\, {\rm d}s +
\int_{\partial E} (\bb \cdot \nn^E) \eint  e_h \, {\rm d}s 
\\
& \lesssim
\beta_E \bigl(\|\eint\|_{L^2(\partial E)} + \|u - \P0 \uint\|_{L^2(\partial E)} \bigr)
\|(I - \P0)e_h\|_{L^2(\partial E)}  +
\int_{\partial E} (\bb \cdot \nn^E) \eint  e_h \, {\rm d}s  
\\
& \lesssim
\beta_E  h_E^{-1}
\bigl(\tri \eint \tri_{1,E} + \tri u - \P0 \uint \tri_{1,E} \bigr)
\|(I - \P0)e_h\|_{0, E} 
 + 
\int_{\partial E} (\bb \cdot \nn^E) \eint  e_h \, {\rm d}s 
\\
& \lesssim
\beta_E  
\bigl(\tri \eint \tri_{1,E} + \tri u - \P0 u \tri_{1,E} \bigr)
\|\nabla(I - \PN)e_h\|_{0, E} 
 + 
\int_{\partial E} (\bb \cdot \nn^E) \eint  e_h \, {\rm d}s 
\\
& \lesssim
\sigma_E h_E^{\reg+1} |u|_{\reg+1, E}
 + 
\int_{\partial E} (\bb \cdot \nn^E) \eint  e_h \, {\rm d}s  
\end{aligned}
\end{equation}
The thesis now follows gathering \eqref{eq:ebA2}, \eqref{eq:ebB2} and \eqref{eq:ebC2} in \eqref{eq:errb2}.
\end{proof}

\begin{remark}\label{rem:chal}
The main difficulty in proving Lemmas \ref{lm:errbp} and \ref{lm:errbb} lays in handling a variable coefficient $\bb$ in the presence of projection operators, without paying a price for small values of $\varepsilon$. For form \eqref{eq:bfp}, we are able to obtain a ``damped'' dependence on $\varepsilon$:  in estimate \eqref{eq:errbpstima} the term $\epsilon^{-1/2}
|\bb|_{[W^1_{\infty}(E)]^2}
h_E^{\reg+2} |u|_{\reg+1,E}$ blows up as $\epsilon\to 0$, but at the same time it is of higher order with respect to $h_E$. Instead, for the new form \eqref{eq:bfb} we are able to obtain full independence from $\varepsilon$.
\end{remark}

We are now ready to prove the convergence results for the proposed VEM SUPG scheme.
The error estimates in Lemmas \ref{lm:epi}-- \ref{lm:errbb} 
 are explicit in the parameters of interest:
the mesh size $h$, the diffusive coefficient $\epsilon$, the advective field $\beta$ and the SUPG parameter $\tau_E$.
In order to simplify the final estimate and to make clearer 
the implications of the convergence results, in the following propositions we include the Sobolev regularity terms for
$u$, $f$ and the normalized norms $\frac{\|\bb\|_{[W^m_p(E)]^2}}{\beta_E}$ in the constant.

\begin{proposition}
\label{prp:bfp}
Under the assumptions \textbf{(A1)} and \textbf{(A2)}, let $u \in V$ be the
solution of equation \eqref{eq:problem-c} and $u_h \in V_h(\Omega_h)$ be the solution of equation \eqref{eq:supg-vem} obtained with the bilinear form $\bfp(\cdot, \cdot)$ in \eqref{eq:bfp}.
Then it holds that
\begin{multline*}
\|u - u_h\|^2_{\supg} \lesssim
\sum_{E \in \Omega_h}
\Theta^E_o
\biggl(
h^{2\reg}_E 
 (\epsilon +  \tau_E \beta_E^2 + \tau_E) + 
\lambda^2_E h_E^{2(\reg+2)} +
\lambda^2_E \beta_E^2 h_E^{2(\reg+1)}  +
\biggr.
\\
\left.
+
\tau_E \epsilon^2   h_E^{2(\reg-1)}  +
\beta_E^2 \frac{h_E^{2(\reg+2)}}{\epsilon}
+
\beta_E^2 \frac{h^{2(2\reg+1)}}{\epsilon}
\right) \,,
\end{multline*}
where the constant $\Theta^E_o$ depends on
$\|u\|_{s+1,E}$, $\|f\|_{s+1,E}$, $\frac{\|\beta\|_{[W^{s+1}_{\infty}(E)]^2}}{\beta_E}$.
\end{proposition}

\begin{proof}
The proof is a direct consequence of Proposition \ref{prp:abstract}, Lemmas \ref{lm:epi}, \ref{lm:errA}, \ref{lm:errB}, \ref{lm:errL}, and \ref{lm:errbp} where \\
(1) we made use of $\sigma_E = \beta_E \lambda_E$, \\
(2) we wrote $|\bb|^2_{[W^1_{\infty}(E)]^2} = \beta_E^2 (|\bb|_{[W^1_{\infty}(E)]^2}/\beta_E)^2$ \\
(3) we estimated the last two terms in \eqref{eq:errbpstima} as follows. \\
The penultimate term is bounded using the Poincar\'e inequality on the domain $\Omega$ 
\[
\begin{aligned}
& \sum_{E \in \Omega_h}|\bb|_{[W^{\reg+1}_{\infty}(E)]^2}
h_E^{2\reg+1} |u|_{\reg+1,E} \|e_h\|_{0,E} 
\leq
\\
& \leq
\left( \sum_{E \in \Omega_h} \beta_E^2 \: |\bb/\beta_E|^2_{[W^{\reg+1}_{\infty}(E)]^2}
h_E^{2(2\reg+1)} |u|^2_{\reg+1,E} \right)^{1/2} 
\left( \sum_{E \in \Omega_h} \|e_h\|^2_{0,E}
\right)^{1/2} 
\\
& \lesssim
\left( \sum_{E \in \Omega_h} \beta_E^2
h_E^{2(2\reg+1)} |u|^2_{\reg+1,E} \right)^{1/2} 
\|e_h\|_{0, \Omega}
\\
& \lesssim
\left( \sum_{E \in \Omega_h} \beta_E^2
h_E^{2(2\reg+1)} |u|^2_{\reg+1,E} \right)^{1/2} 
\|\nabla e_h\|_{0,\Omega}
\\
&\lesssim
\left( \sum_{E \in \Omega_h} \beta_E^2
\frac{h_E^{2(2\reg+1)}}{\epsilon} |u|^2_{\reg+1,E} \right)^{1/2} 
\|e_h\|_{\supg} \,.
\end{aligned}
\]
For the last term, noticing that $\eint, e_h \in V$, it holds that
\begin{equation}
\label{eq:boundarycan}
\sum_{E \in \Omega_h} \int_{\partial E} (\bb \cdot \nn^E) \eint e_h \, {\rm d}s = 0 \,.
\end{equation}
\end{proof}

\begin{proposition}
\label{prp:bfb}
Under the assumptions \textbf{(A1)} and \textbf{(A2)}, let $u \in V$ be the
solution of equation \eqref{eq:problem-c} and $u_h \in V_h(\Omega_h)$ be the solution of equation \eqref{eq:supg-vem} obtained with the bilinear form $\bfb(\cdot, \cdot)$ in \eqref{eq:bfb}.
Then it holds that
\[
\begin{aligned}
\|u - u_h\|^2_{\supg} &\lesssim
\sum_{E \in \Omega_h} \Theta^E_{\partial}
\left(
h^{2\reg}_E 
 (\epsilon +  \tau_E \beta_E^2 + \tau_E) + 
\lambda^2_E h_E^{2(\reg+2)} +
\beta_E^2 \lambda^2_E h_E^{2(\reg+1)}  +
\tau_E \epsilon^2   h_E^{2(\reg-1)} \right)
\end{aligned}
\]
where the constant $\Theta^E_{\partial}$ depends on
$\|u\|_{s+1,E}$, $\|f\|_{s+1,E}$, $\frac{\|\beta\|_{[W^{s+1}_{\infty}(E)]^2}}{\beta_E}$.
\end{proposition}

\begin{proof}
The proof follows from Proposition \ref{prp:abstract}, Lemmas \ref{lm:epi}, \ref{lm:errA}, \ref{lm:errB}, \ref{lm:errL}, \ref{lm:errbb} and equation \eqref{eq:boundarycan}, also recalling that $\sigma_E = \beta_E^2 \lambda^2_E$.
\end{proof}

It is well known that in order to obtain a stable and optimal convergent method both in the convective and in the diffusion dominated regime the SUPG parameter $\tau_E$ has to be chosen in accordance with
\[
\tau_E \simeq \min \left\{ \frac{h_E}{\beta_E}, \frac{h_E^2}{\epsilon} \right\} \,.
\]
Let us analyse the asymptotic order of convergence for the two versions of VEM scheme in both regimes (where we recall $\beta_E \lesssim 1$ for all $E \in \Omega_h$ due to the scaling choice {\bf (A0)}).

\noindent
$\bullet$ \texttt{convection dominated regime} $\epsilon \ll h_E \beta_E$:   
$\tau_E = \beta_E^{-1} h_E$,  
$\lambda^2_E = \beta_E^{-1} h_E^{-1}$,

\begin{itemize}
\item [-] \texttt{form} $\bfp(\cdot, \cdot)$
\[
\begin{aligned}
\|u - u_h\|^2_{\supg} &\lesssim
\sum_{E \in \Omega_h} 
\left(
h^{2\reg+1}_E 
 (\beta_E +  \beta_E^{-1}) + 
\beta_E^{-1} h_E^{2\reg+3} +
\epsilon h^{2\reg}_E  +
\beta_E^{-1} \epsilon^2   h_E^{2\reg-1}  + \right.
\\
\bigl .
& \qquad +\beta_E^2 \epsilon^{-1}  h_E^{2(\reg+2)} +  \beta_E^2 \epsilon^{-1}  h_E^{2(2\reg+1)} \bigr)
= O \left(h^{2\reg+1} ( 1 + \epsilon^{-1}  h^3) \right)\,;
\end{aligned}
\]
\item [-] \texttt{form} $\bfb(\cdot, \cdot)$
\[
\begin{aligned}
\|u - u_h\|^2_{\supg} &\lesssim
\sum_{E \in \Omega_h} 
\left(
h^{2\reg+1}_E 
 (\beta_E +  \beta_E^{-1}) + 
\beta_E^{-1} h_E^{2\reg+3} +
\epsilon h^{2\reg}_E  +
\beta_E^{-1} \epsilon^2   h_E^{2\reg-1} \right) = O(h^{2\reg+1})\,;
\end{aligned}
\]
\end{itemize}

\noindent
$\bullet$ \texttt{diffusion dominated regime} $\beta_E h_E \ll \epsilon$:
$\tau_E = h_E^2\epsilon^{-1}$,
$\lambda^2_E = \epsilon^{-1}$, 

\begin{itemize}
\item [-] \texttt{form} $\bfp(\cdot, \cdot)$
\[
\begin{aligned}
\|u - u_h\|^2_{\supg} &\lesssim
\sum_{E \in \Omega_h} 
\left(
 \epsilon h^{2\reg}_E +
\beta_E^2\epsilon^{-1} h_E^{2(\reg +1)} +
\epsilon^{-1} h_E^{2(\reg +2)}
\right)
= O(\epsilon h^{2\reg} ) \,;
\end{aligned}
\]
\item [-] \texttt{form} $\bfb(\cdot, \cdot)$
\[
\begin{aligned}
\|u - u_h\|^2_{\supg} &\lesssim
\sum_{E \in \Omega_h} 
\left(
\epsilon h^{2\reg}_E +
\beta_E^2\epsilon^{-1} h_E^{2(\reg +1)} +
\epsilon^{-1} h_E^{2(\reg +2)}
\right)
= O(\epsilon h^{2\reg} ) \,.
\end{aligned}
\]
\end{itemize}
We conclude that in the diffusion dominated regime both schemes yield the optimal rate of convergence.
In the convection dominated regime only the scheme derived from the bilinear forms $\bfb(\cdot, \cdot)$ has the optimal rate of convergence. For the scheme derived from $\bfp(\cdot, \cdot)$ the error is polluted by $\epsilon^{-1}$. Nevertheless, we stress that such a factor appears in front of the ``higher'' order term $h^3$, 
therefore the influence of the diffusion coefficient is strongly reduced.

\input{numExe}

\section*{Acknowledgements}

The authors L. BdV, F. D. and G. V. were partially supported by the European Research Council through
the H2020 Consolidator Grant (grant no. 681162) CAVE, ``Challenges and Advancements in Virtual Elements''. This support is gratefully acknowledged. The authors L. BdV and C. L. were partially supported by the italian PRIN 2017 grant
``Virtual Element Methods: Analysis and Applications''. This support is gratefully acknowledged.

\addcontentsline{toc}{section}{\refname}
\bibliographystyle{plain}
\bibliography{biblio}

\end{document}

%% file: numExe.tex
\section{Numerical experiments}

In this section we numerically validate the proposed methods by means of the following model problem.

\paragraph{Model problem.}
We consider a family of problems in the unit square $\Omega=(0,\,1)^2$, one per each choice of the parameter $\epsilon$. We select the advection term as
$$
\bb(x,\,y) := \left[\begin{array}{r}
-2\,\pi\,\sin(\pi\,(x+2\,y))\\
\pi\,\sin(\pi\,(x+2\,y))
\end{array}\right]\,.
$$

We choose the boundary conditions and the source term (which turns out to depend on $\epsilon$) in such a way that the analytical solution is always the function
$$
u(x,\,y) := \sin(\pi\,x)\sin(\pi\,y)\,.
$$

Guided by the definition of the $||\cdot||_\supg$ norm (cf. \eqref{eq:loc_nom_def} and \eqref{eq:glb_norm_def}), by the error estimates of Propositions \ref{prp:bfp}--\ref{prp:bfb}, and noticing that the discrete solution $u_h\in V_h(\Omega_h)$ is not explicitly pointwise available, the following error quantities will be considered.   

\begin{itemize}
 \item \textbf{$H^1-$seminorm error}
 $$
 e_{H^1} := \sqrt{\sum_{E\in\mathcal{T}_h}\left\|\nabla(u-\Pi_k^\nabla u_h)\right\|^2_{0,E}}\,;
 $$
 
 \item \textbf{convective norm error}
 $$
 e_{\mathcal{C}} := \sqrt{\sum_{E\in\mathcal{T}_h}\left(\epsilon\,\bigg\| \nabla(u-\Pi_k^\nabla u_h)\bigg\|^2_{0,E} + 
 \tau_E\bigg\|\bb\cdot\nabla(u-\Pi_k^\nabla u_h)\bigg\|^2_{0,E}\right)}\,.
 $$
\end{itemize}

\noindent
As far as the mesh types are concerned, we take the following:
\begin{itemize}
 \item \texttt{quad:} a mesh composed by structured quadrilaterals;
 \item \texttt{tria:} a Delaunay triangulation of the unit square;
 \item \texttt{voro:} a centroidal Voronoi tessellation of the unit square where the cells' shape is optimized via a Lloyd algorithm
 \item \texttt{rand:} a voronoi tessellation of the unit square where the cell shapes are not optimized.
\end{itemize}
In Figure~\ref{fig:meshes} we show an example of such meshes, and we also remark that the former two types can be used in connection with a standard finite element procedure, contrary to the latter two.


\begin{figure}[!htb]
\begin{center}
\begin{tabular}{ccc}
\texttt{quad} & &\texttt{tria} \\
\includegraphics[width=0.35\textwidth]{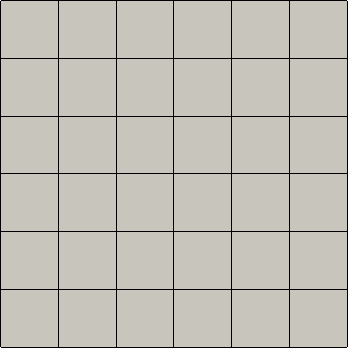} &\phantom{mm}&
\includegraphics[width=0.35\textwidth]{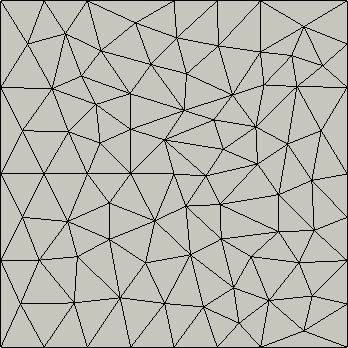} \\
\texttt{voro} & &\texttt{rand} \\
\includegraphics[width=0.35\textwidth]{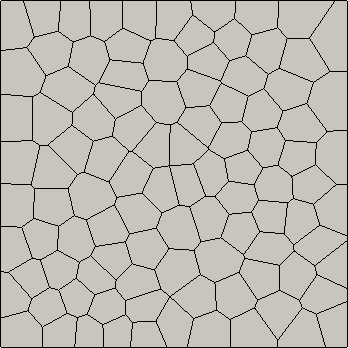} &\phantom{mm}&
\includegraphics[width=0.35\textwidth]{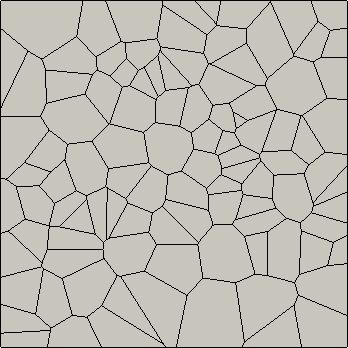} \\
\end{tabular}
\end{center}
\caption{Example of meshes used for the present test case.}
\label{fig:meshes}
\end{figure}

\noindent Moreover, since we are interested in a robustness analysis with respect to the diffusion parameter, for each mesh sequence we take 
$$
\epsilon = 10^{-3}\ \text{and}\  10^{-6}.
$$

\paragraph{Effect of the SUPG stabilization.} Before assessing the convergence properties of the proposed methods, we check the effect of inserting the SUPG term in the variational formulation of the problem.  Here we focus on the first choice for the convective bilinear form, without the algebraical skew-symmetrization; namely we here use form \eqref{eq:bfp}. 
 
\begin{figure}[!htb]
\begin{center}
\begin{tabular}{cc}
\multicolumn{2}{c}{\texttt{SUPG}} \\
\includegraphics[width=0.47\textwidth]{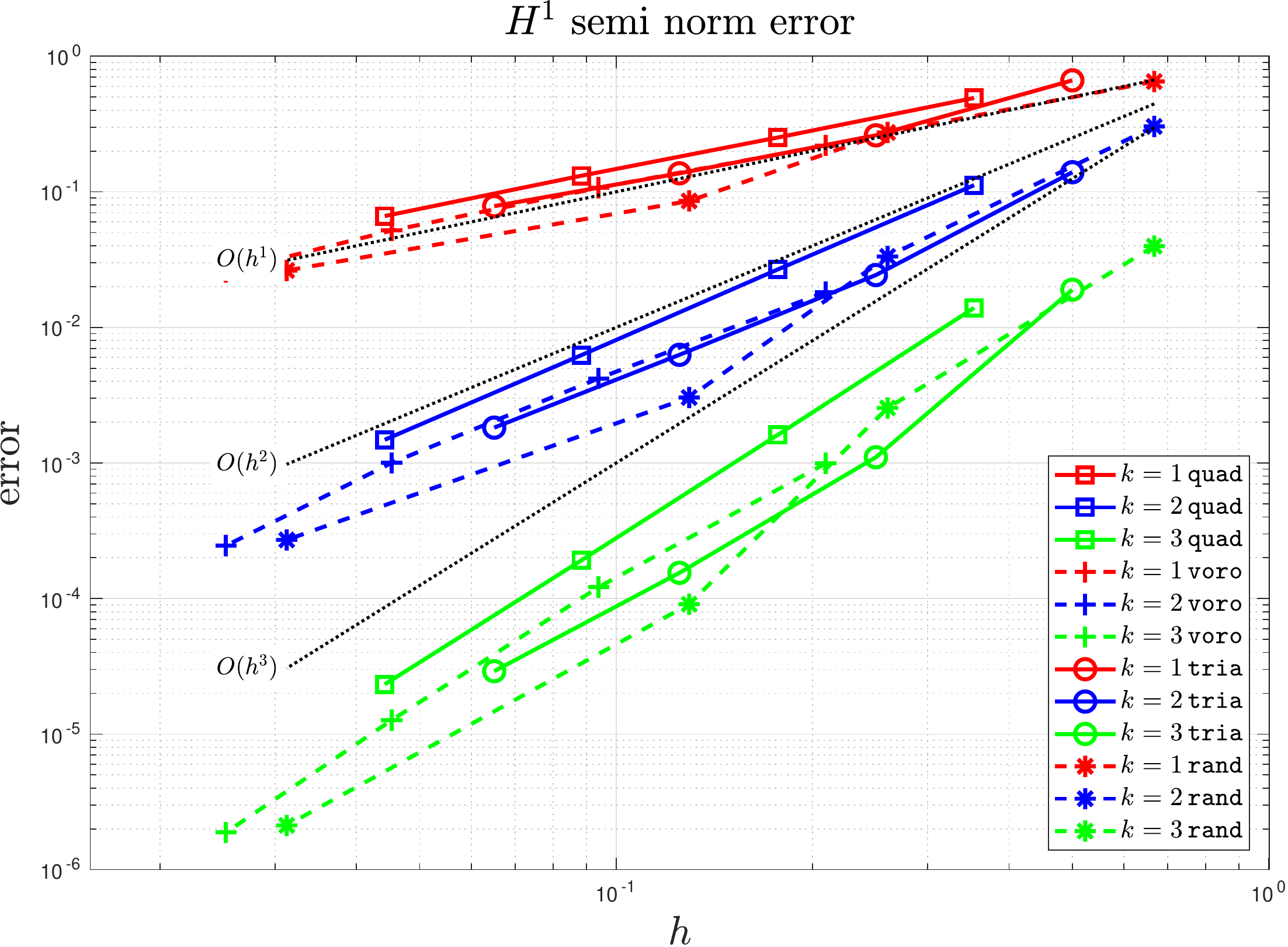} &
\includegraphics[width=0.47\textwidth]{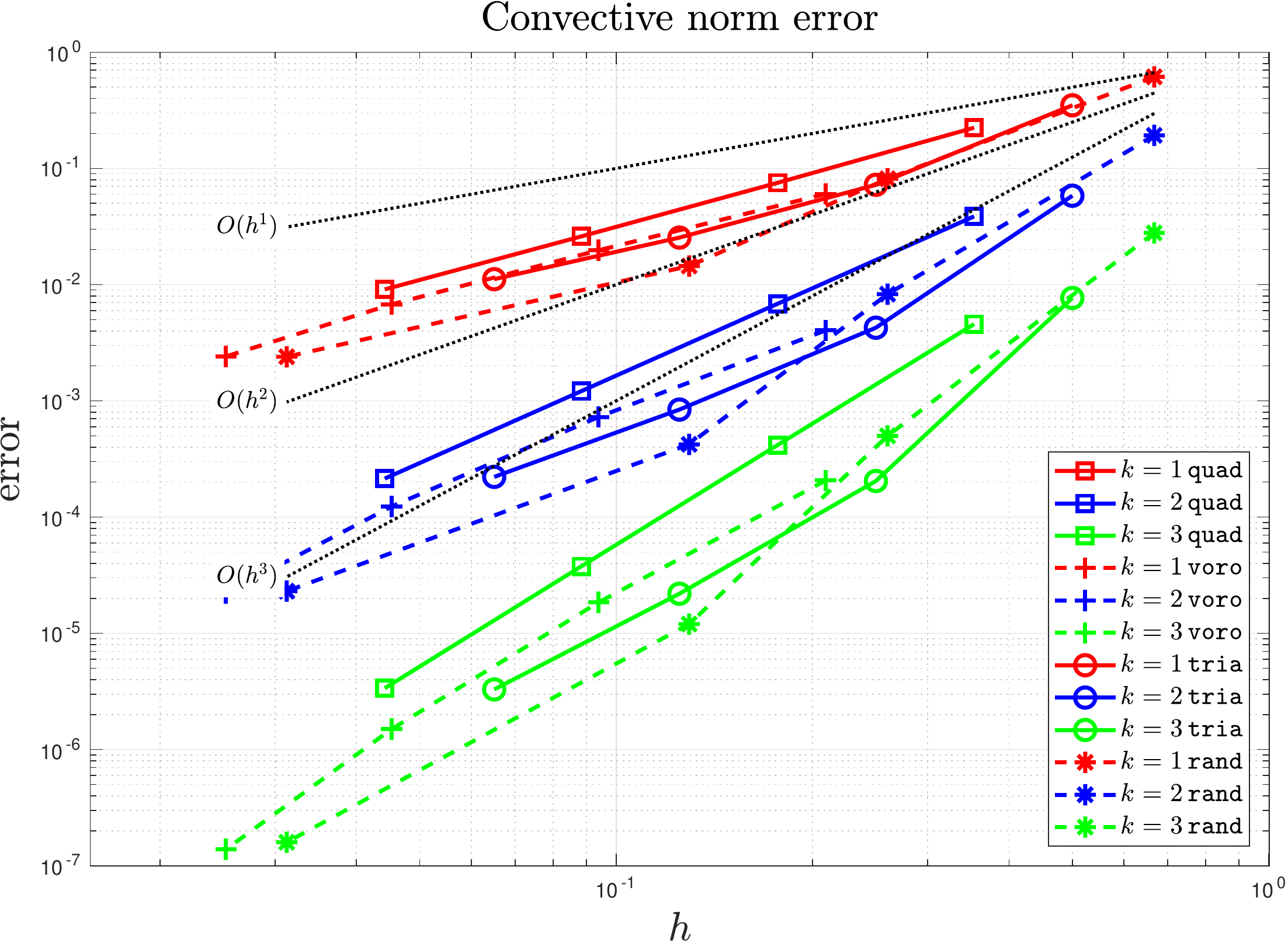} \\
\multicolumn{2}{c}{\texttt{NONE}} \\
\includegraphics[width=0.47\textwidth]{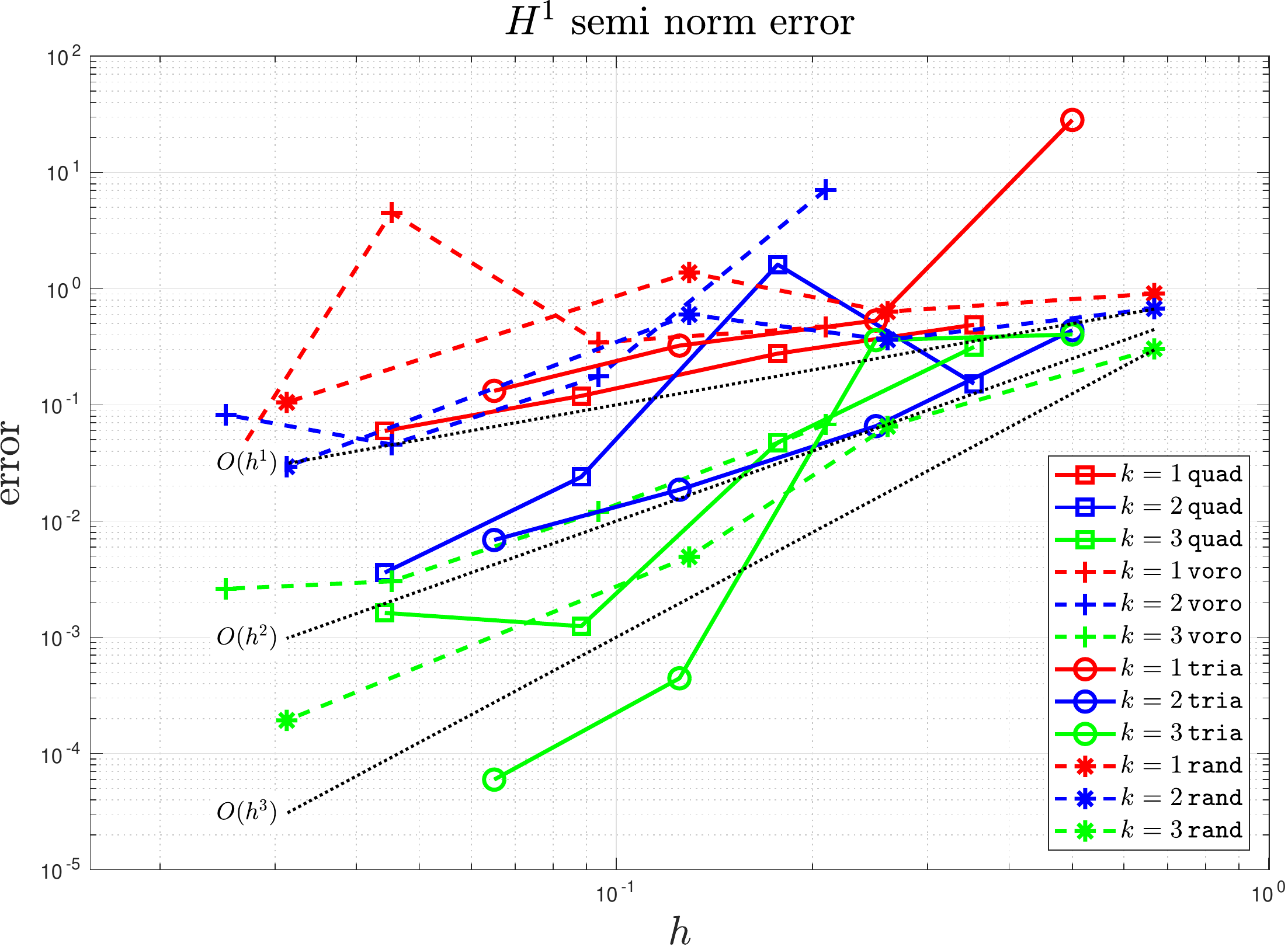} &
\includegraphics[width=0.47\textwidth]{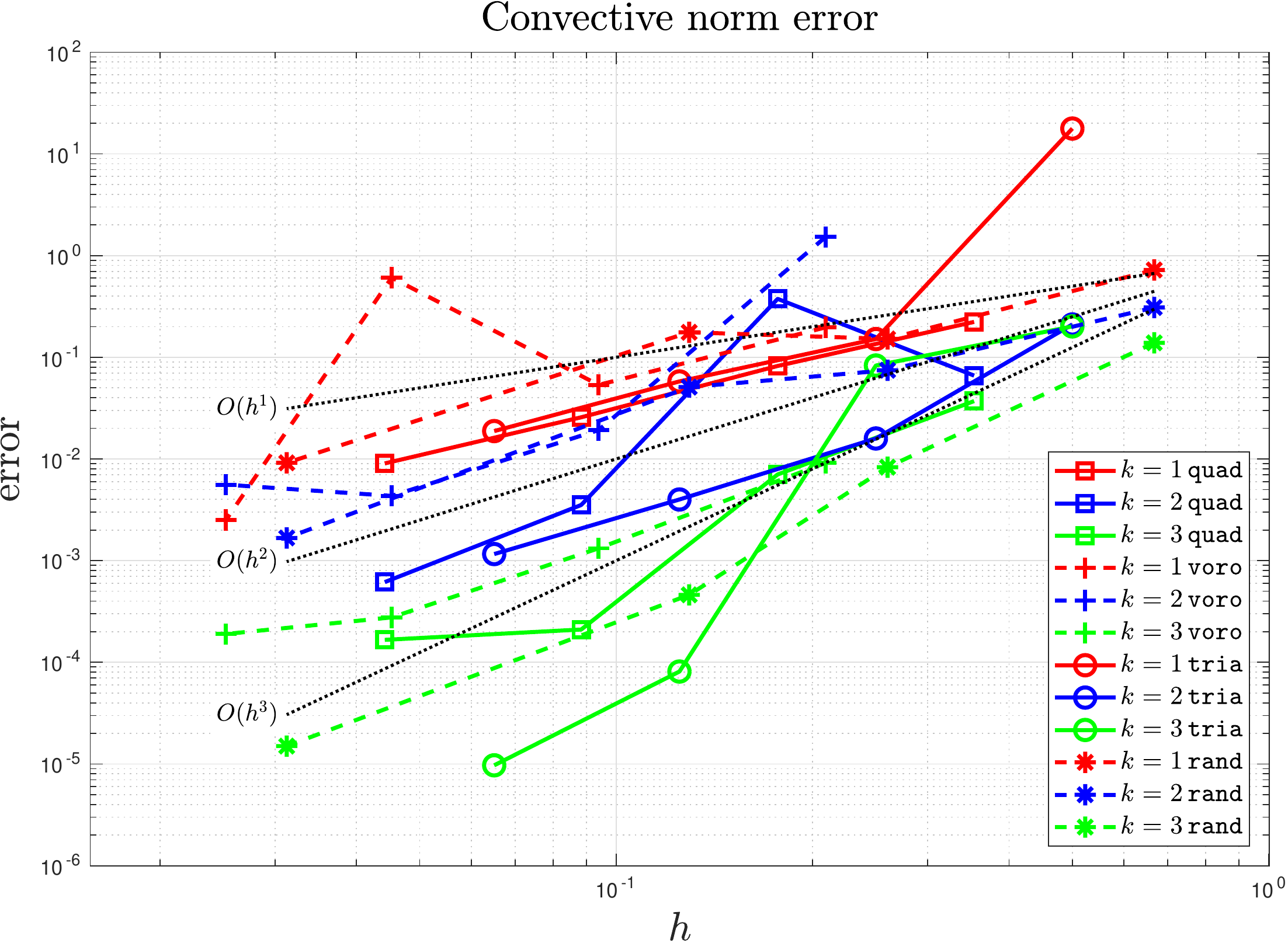} \\
\end{tabular}
\end{center}
\caption{Effect of the SUPG term on the convergence histories: the case $\epsilon=1e-03$.}
\label{fig:convKappaM3}
\end{figure}

\begin{figure}[!htb]
\begin{center}
\begin{tabular}{cc}
\multicolumn{2}{c}{\texttt{SUPG}} \\
\includegraphics[width=0.47\textwidth]{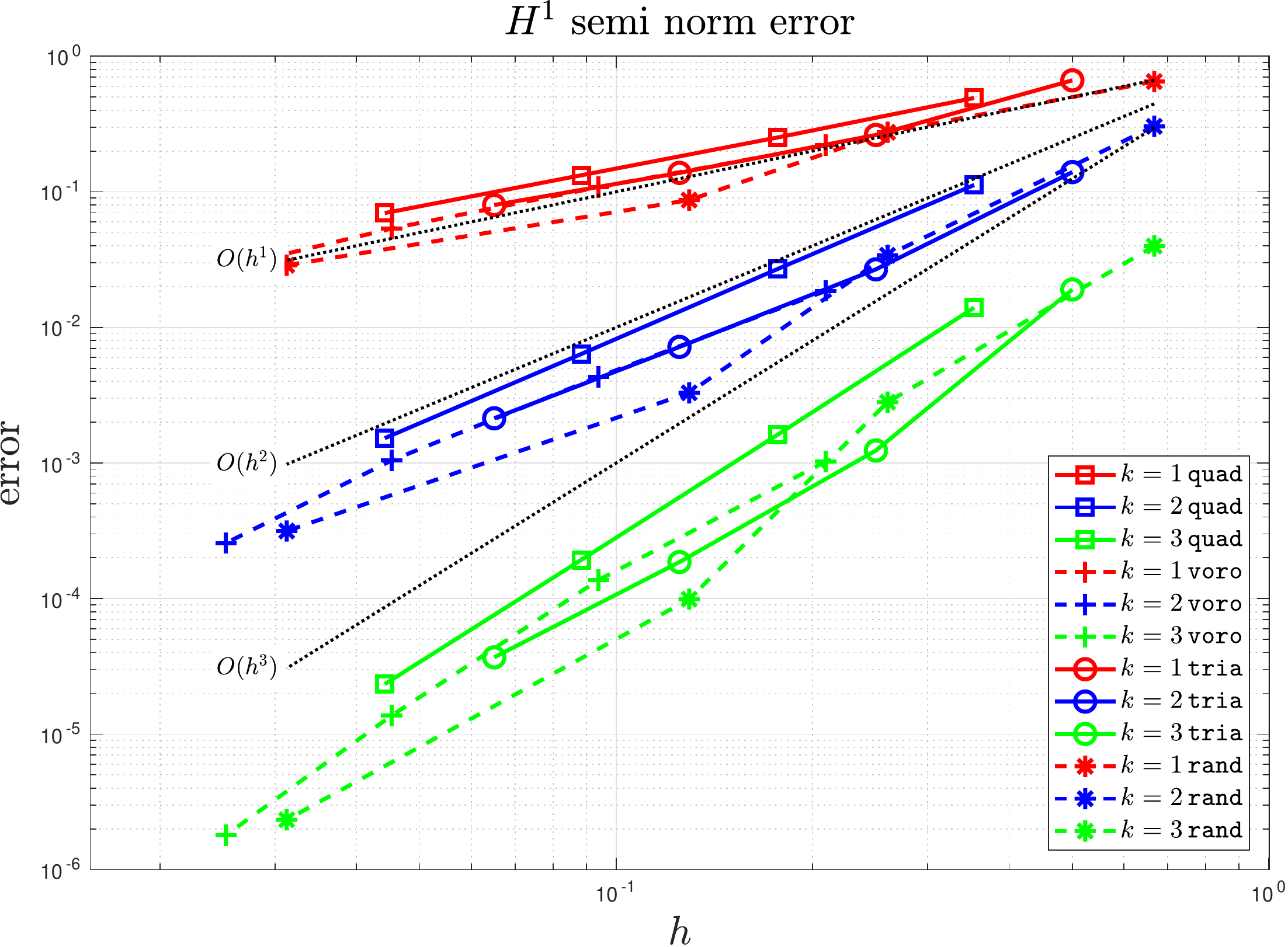} &
\includegraphics[width=0.47\textwidth]{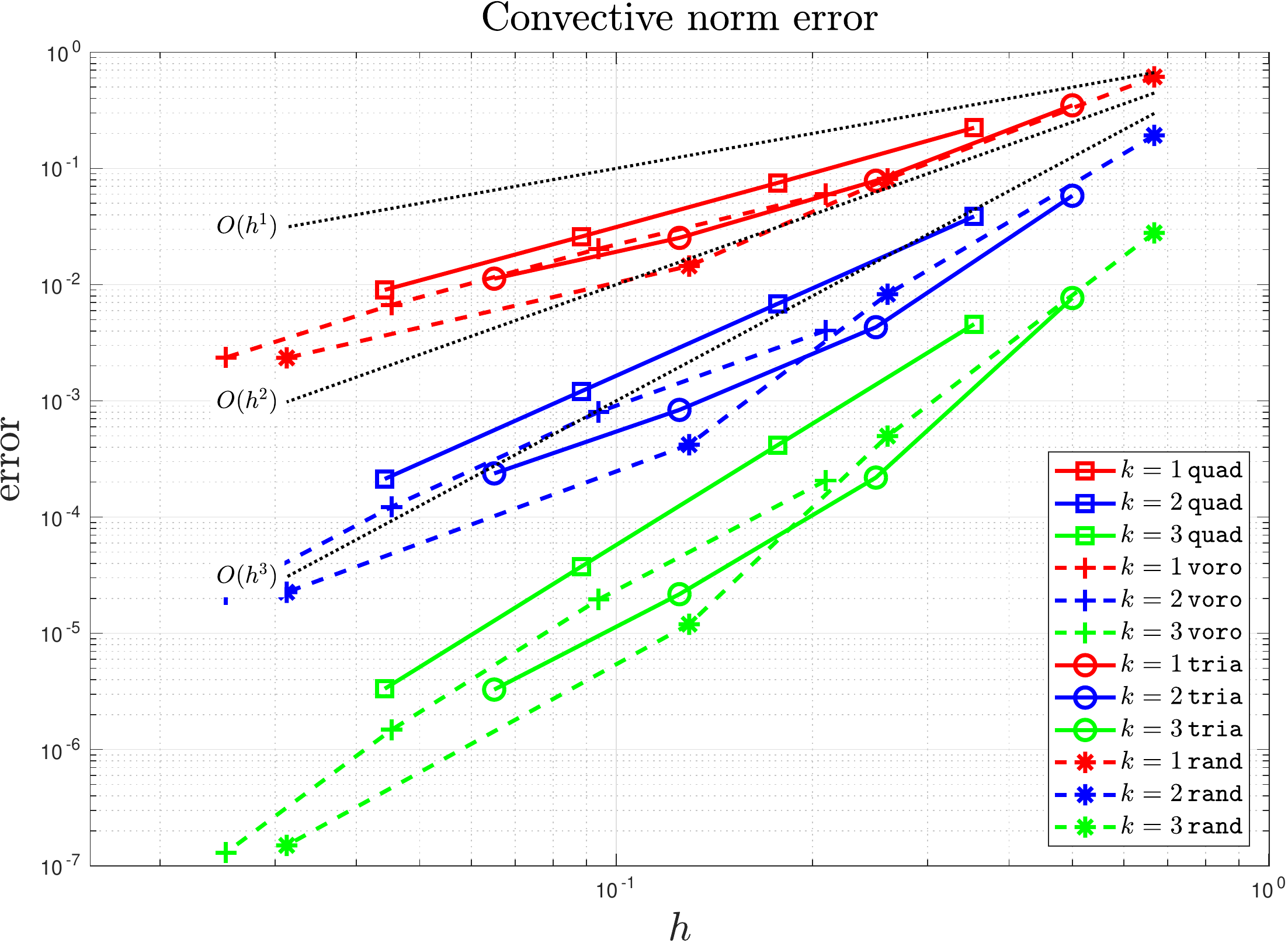} \\
\multicolumn{2}{c}{\texttt{NONE}} \\
\includegraphics[width=0.47\textwidth]{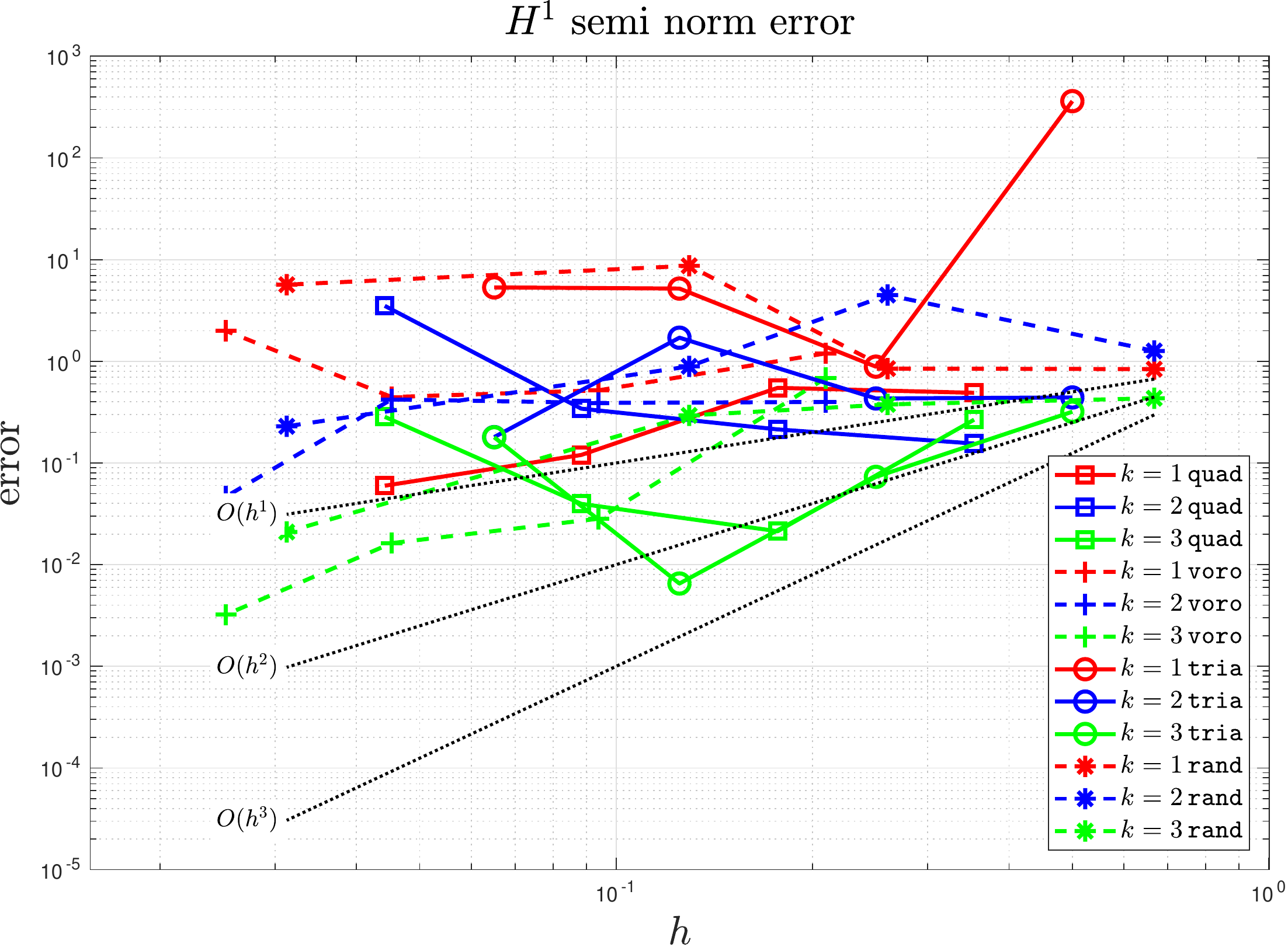} &
\includegraphics[width=0.47\textwidth]{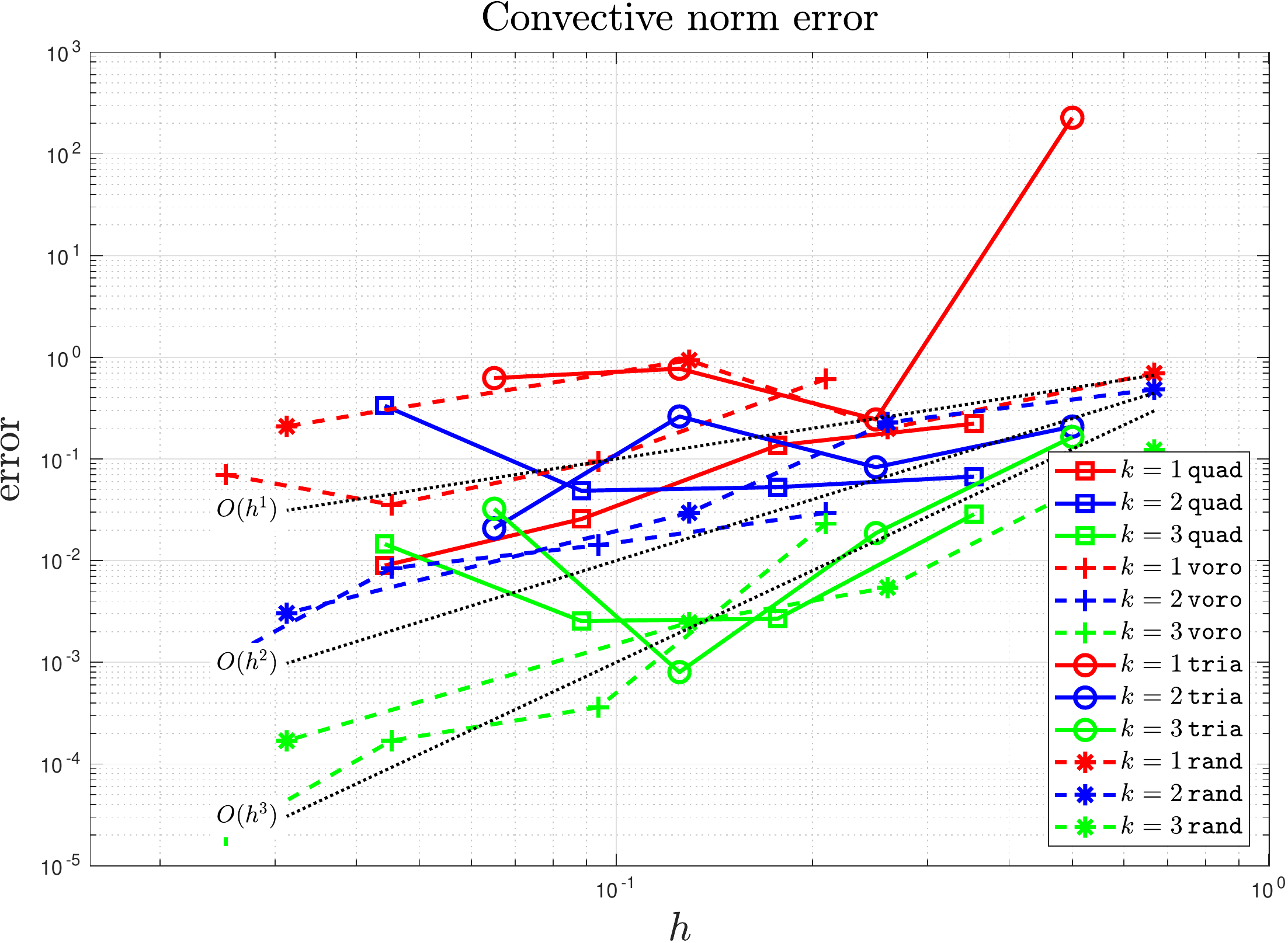} \\
\end{tabular}
\end{center}
\caption{Effect of the SUPG term on the convergence histories: the case $\epsilon=1e-06$.}
\label{fig:convKappaM6}
\end{figure}

In Figures~\ref{fig:convKappaM3} and~\ref{fig:convKappaM6} we show the convergence graphs for different 
approximation degrees $k$ and for each type of the meshes. The cases with and without the SUPG term are labelled as \texttt{SUPG} and \texttt{NONE}, respectively. As expected, dropping the SUPG term clearly deteriorates the quality of the discrete solutions for all approximstion degree $k$, the degeneration being heavier for smaller $\epsilon$. We do not report the results for ``large'' $\epsilon$, but we have experienced that, of course, both the \texttt{SUPG} and the \texttt{NONE} approaches give similar outcomes, attaining the correct convergence rate. 

Incidentally, we remark that Figures~\ref{fig:convKappaM3} and~\ref{fig:convKappaM6} also display the robustness of the present SUPG virtual element approach with respect to element shape and distortion. In fact, given an approximation degree $k$, the convergence histories are rather similar for all the mesh types, even though the \texttt{rand} meshes contain some odd shaped elements.

\paragraph{Different discrete convective bilinear forms.}
We now present the numerical results for different variants of the convective term approximation. 
More precisely, we consider four types of (local) discrete convective terms: 
$b_{o,h}^E(\cdot,\cdot)$ and $b_{\partial,h}^E(\cdot,\cdot)$, and their skew-symmetric counterparts, see \eqref{eq:bfp}, \eqref{eq:bfb} and \eqref{eq:bskewEh}.
We refer to such approximations as \texttt{orig} and \texttt{boun}, respectively, while 
we label their skew-symmetric counterparts as \texttt{origSkew} and \texttt{bounSkew} (the ones covered by the theoretical analysis developed in Section \ref{sec:theory}). 

To avoid instabilities in the convection dominated regime, in what follows we always take advantage of the SUPG stabilization.

In Figures~\ref{fig:convKappaM3Bdiff} and~\ref{fig:convKappaM6Bdiff} 
we collect the results with the diffusion coefficients $\epsilon=10^{-3}$ and $\epsilon=10^{-6}$, respectively.
We notice that the convergence rate of both error norms $e_{\mathcal{C}}$ and $ e_{H^1}$ are the expected ones, and they are robust in the parameter $\epsilon$.
When we consider the $H^1$-seminorm error $e_{H^1}$, we observe that the \texttt{origSkew} behaves slightly worst than the other three discrete forms, especially for the approximation degrees $k=2,3$, and independently of the values of $\epsilon$. 

\begin{figure}[!htb]
\begin{center}
\begin{tabular}{cc}
\multicolumn{2}{c}{\texttt{quad}} \\
\includegraphics[width=0.47\textwidth]{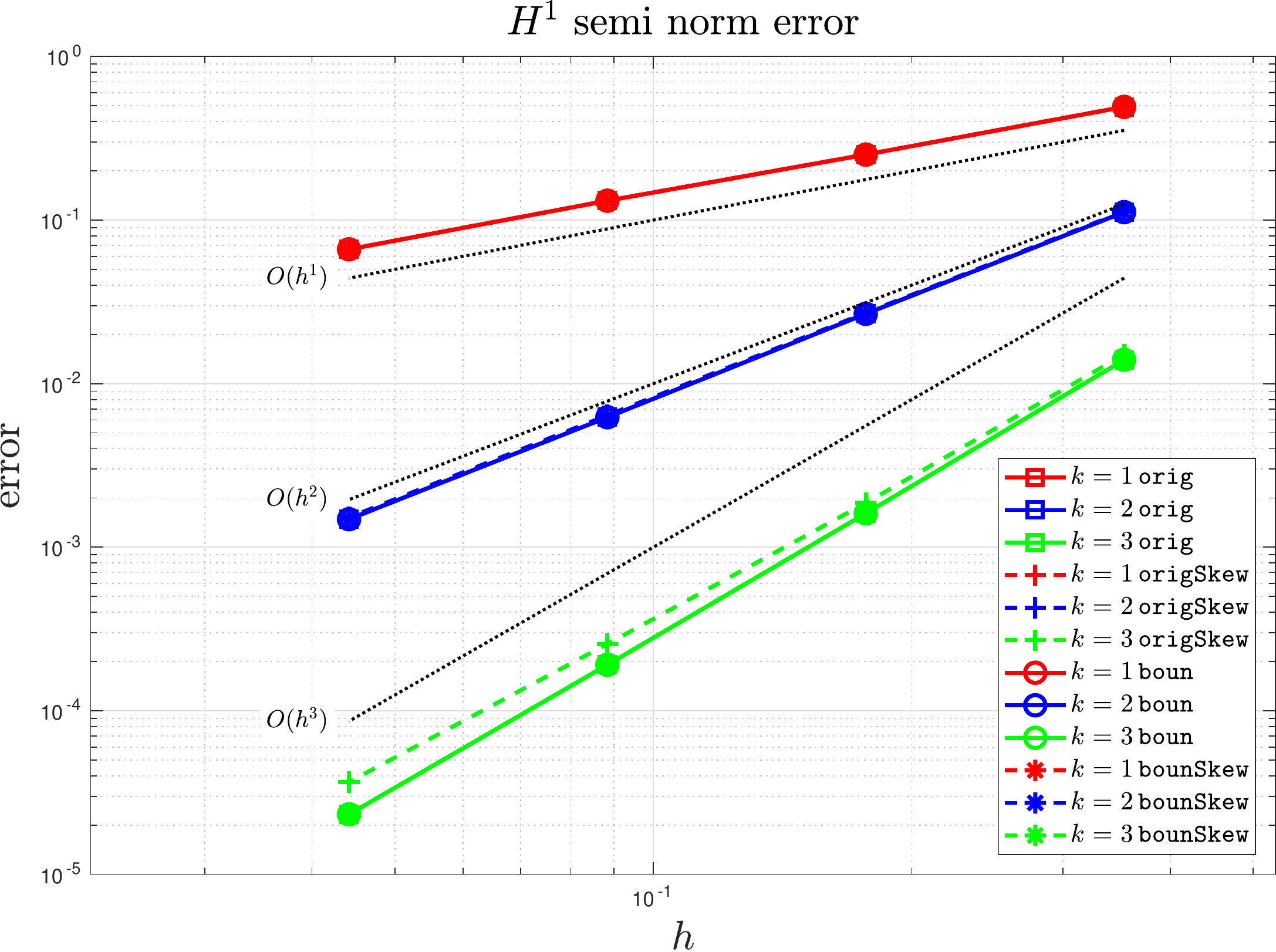} &
\includegraphics[width=0.47\textwidth]{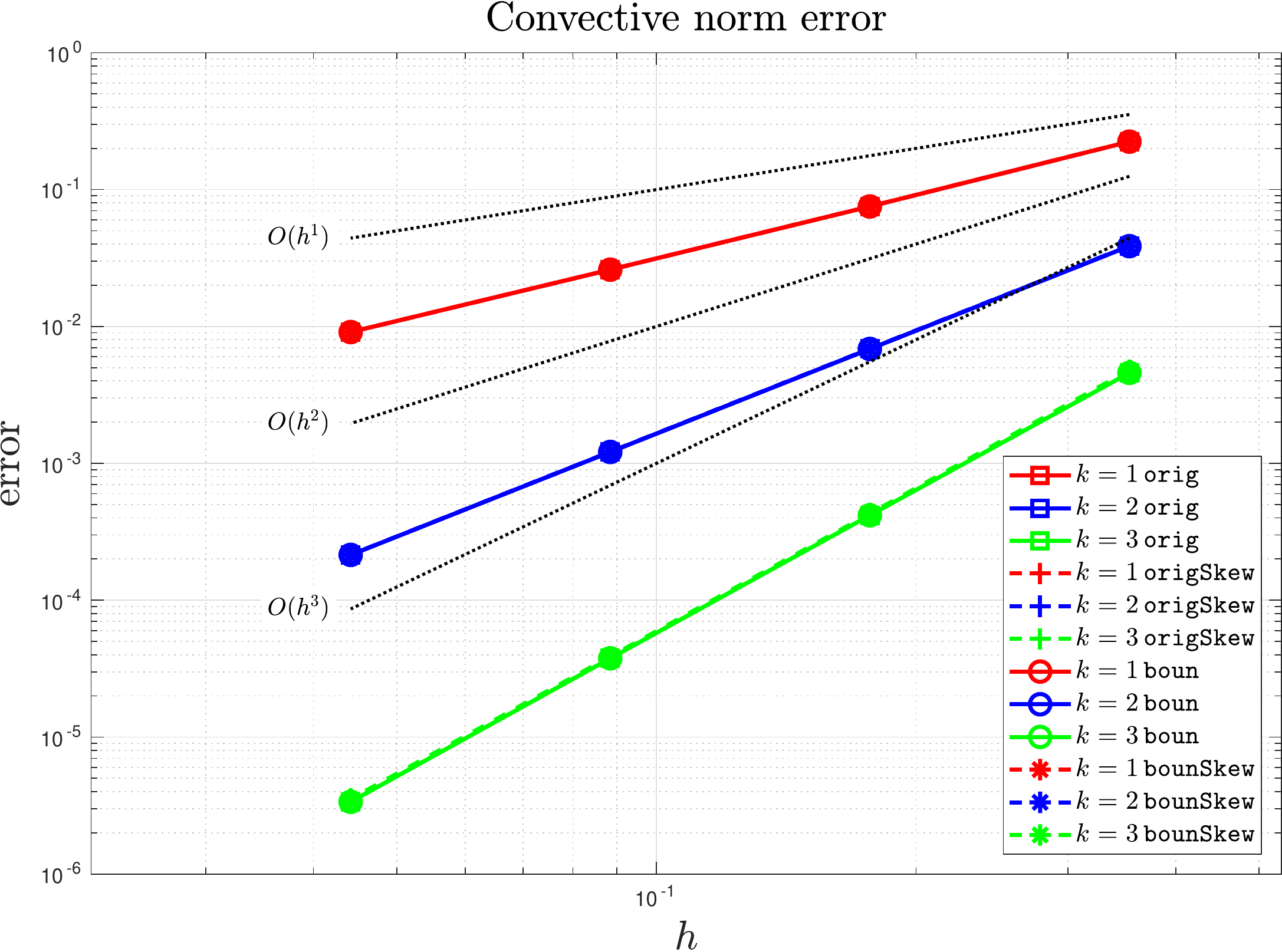} \\
\multicolumn{2}{c}{\texttt{tria}} \\
\includegraphics[width=0.47\textwidth]{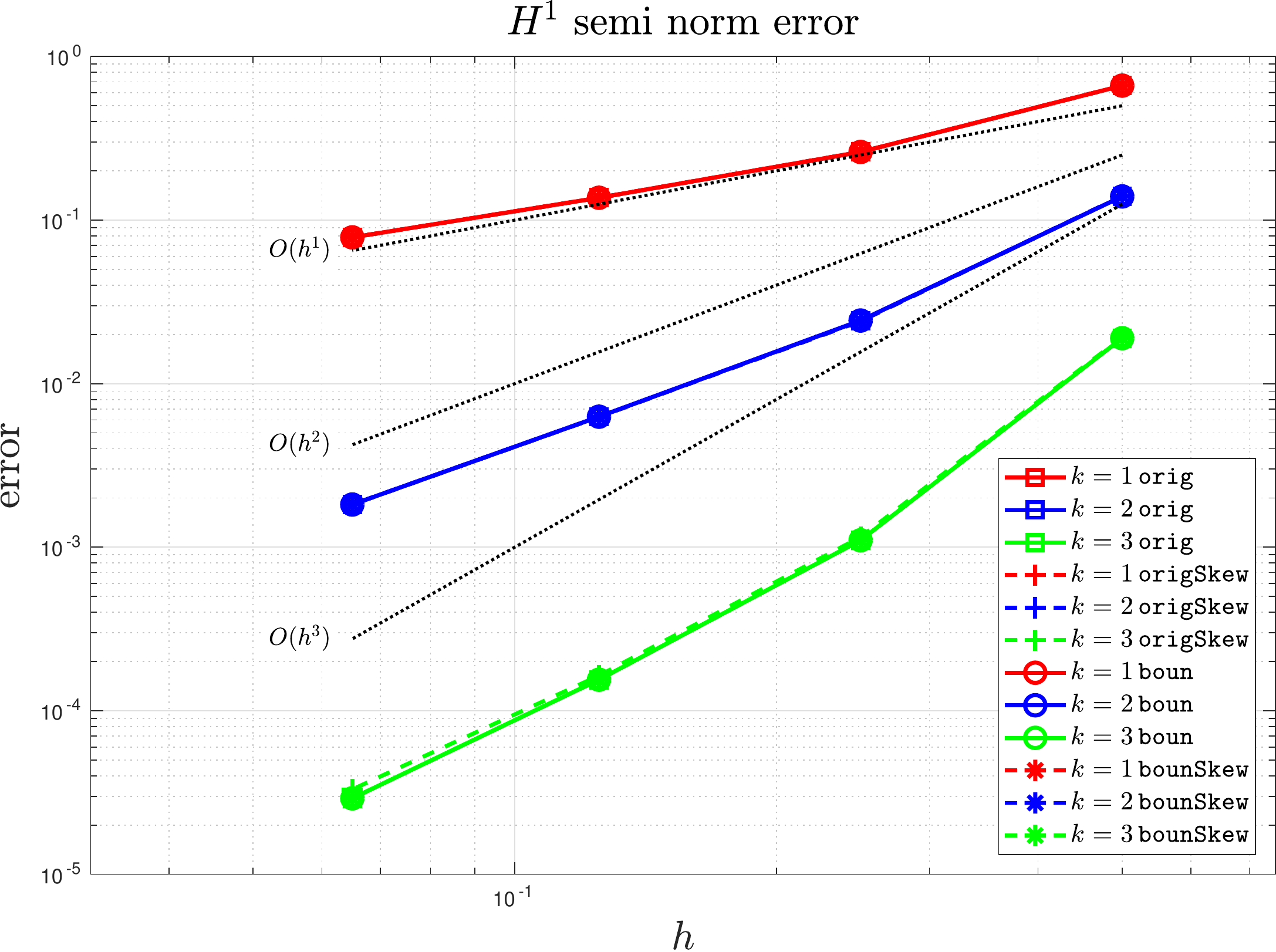} &
\includegraphics[width=0.47\textwidth]{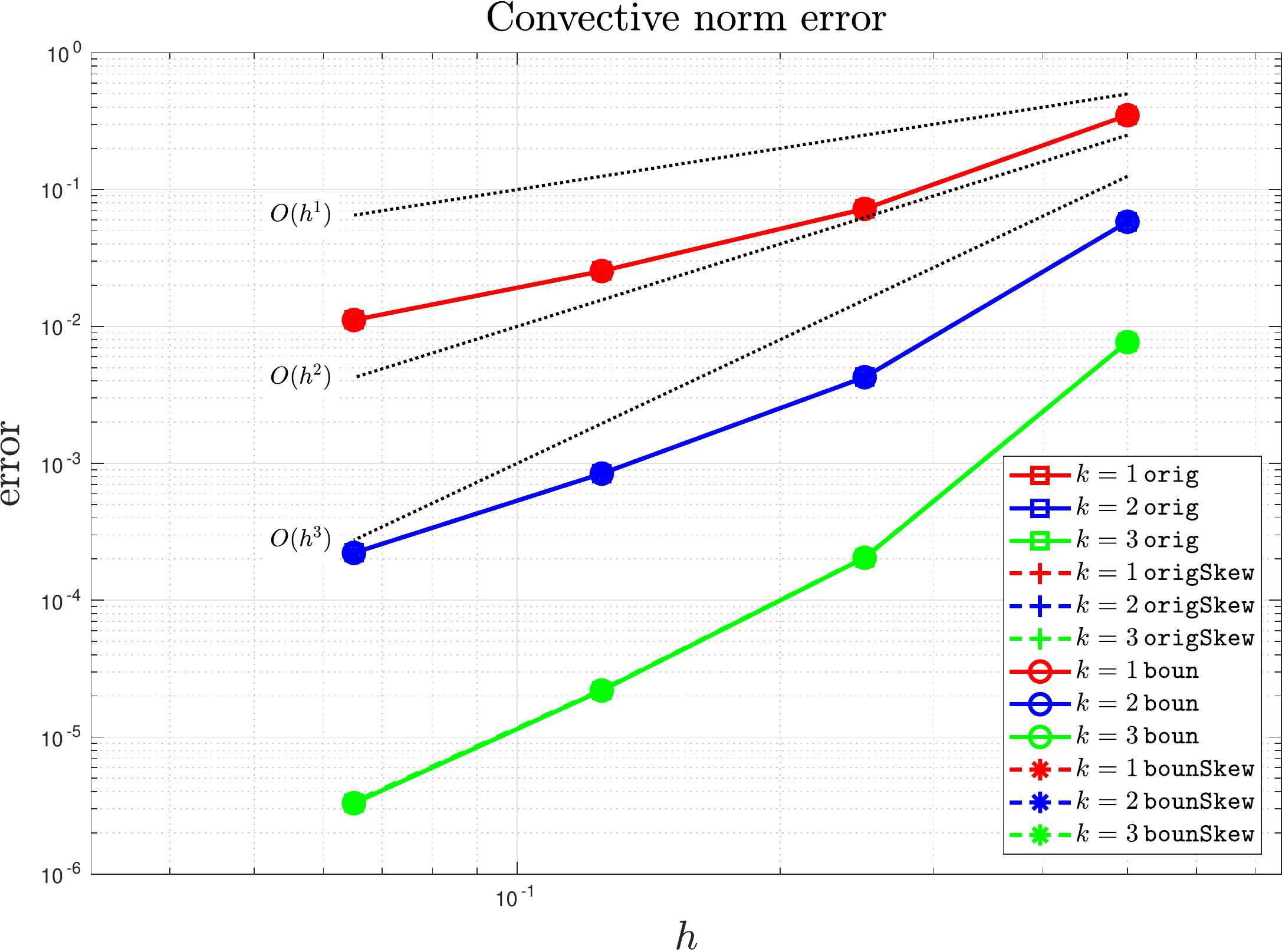} \\
\multicolumn{2}{c}{\texttt{voro}} \\
\includegraphics[width=0.47\textwidth]{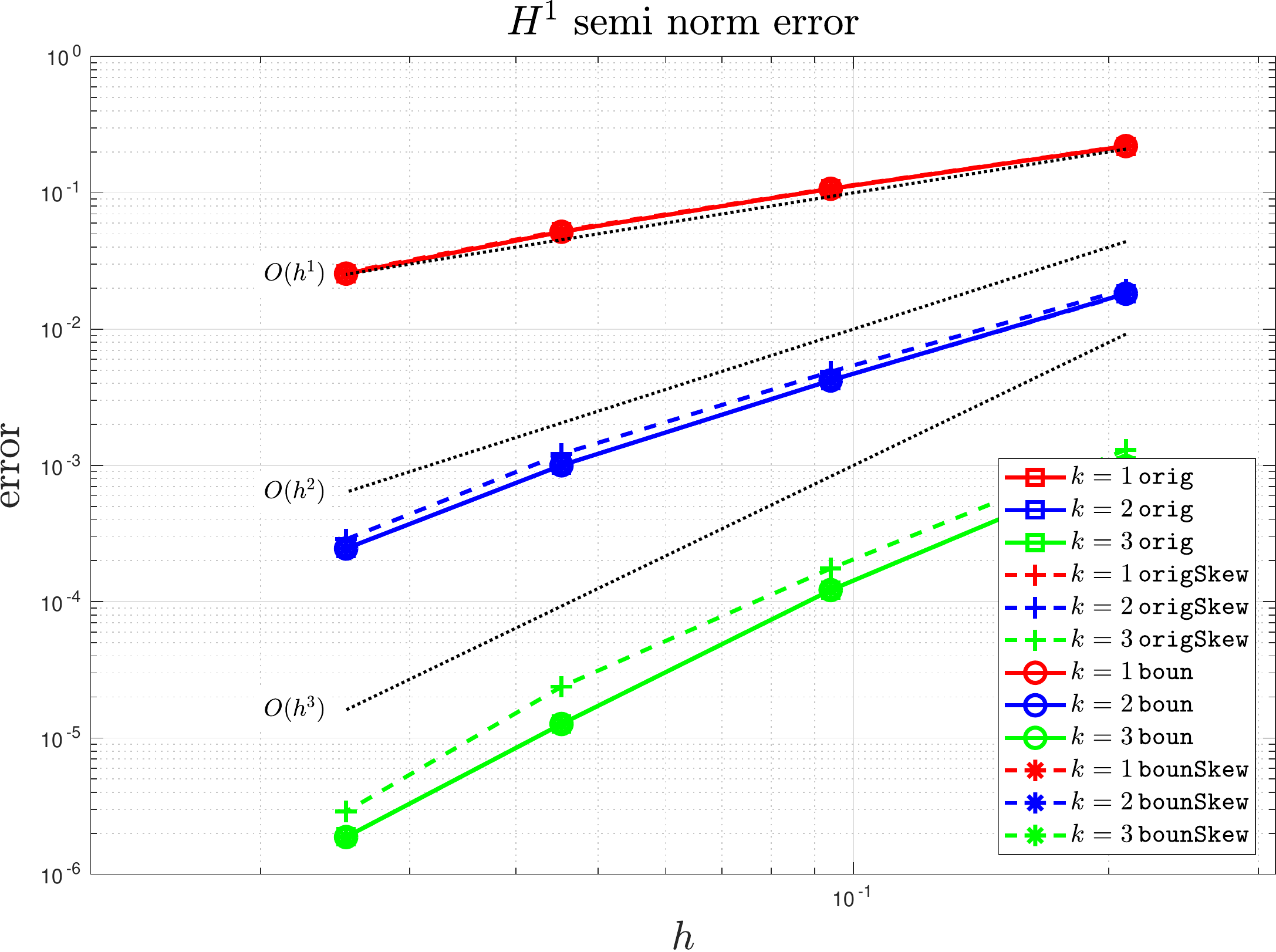} &
\includegraphics[width=0.47\textwidth]{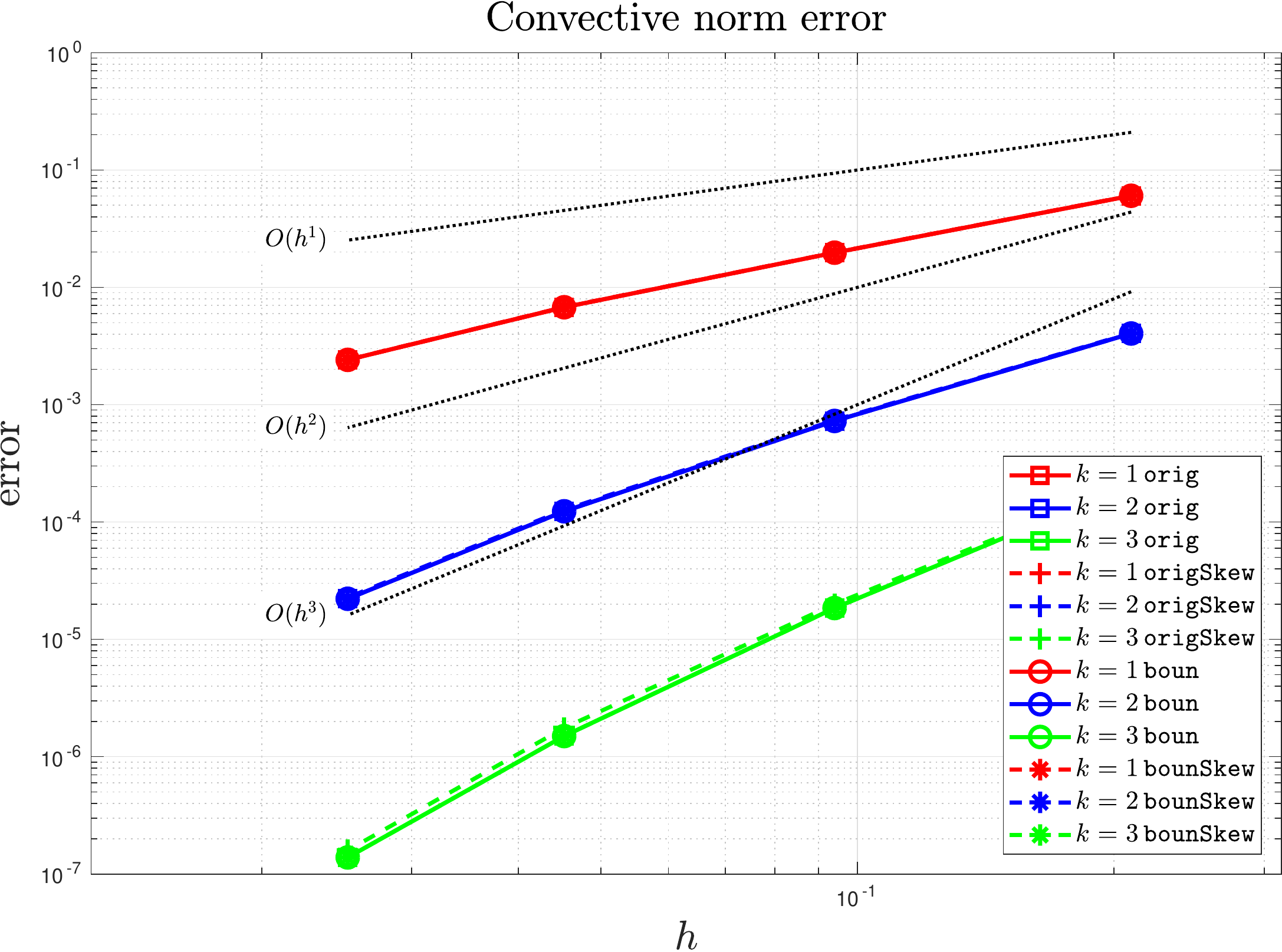} \\
\multicolumn{2}{c}{\texttt{rand}} \\
\includegraphics[width=0.47\textwidth]{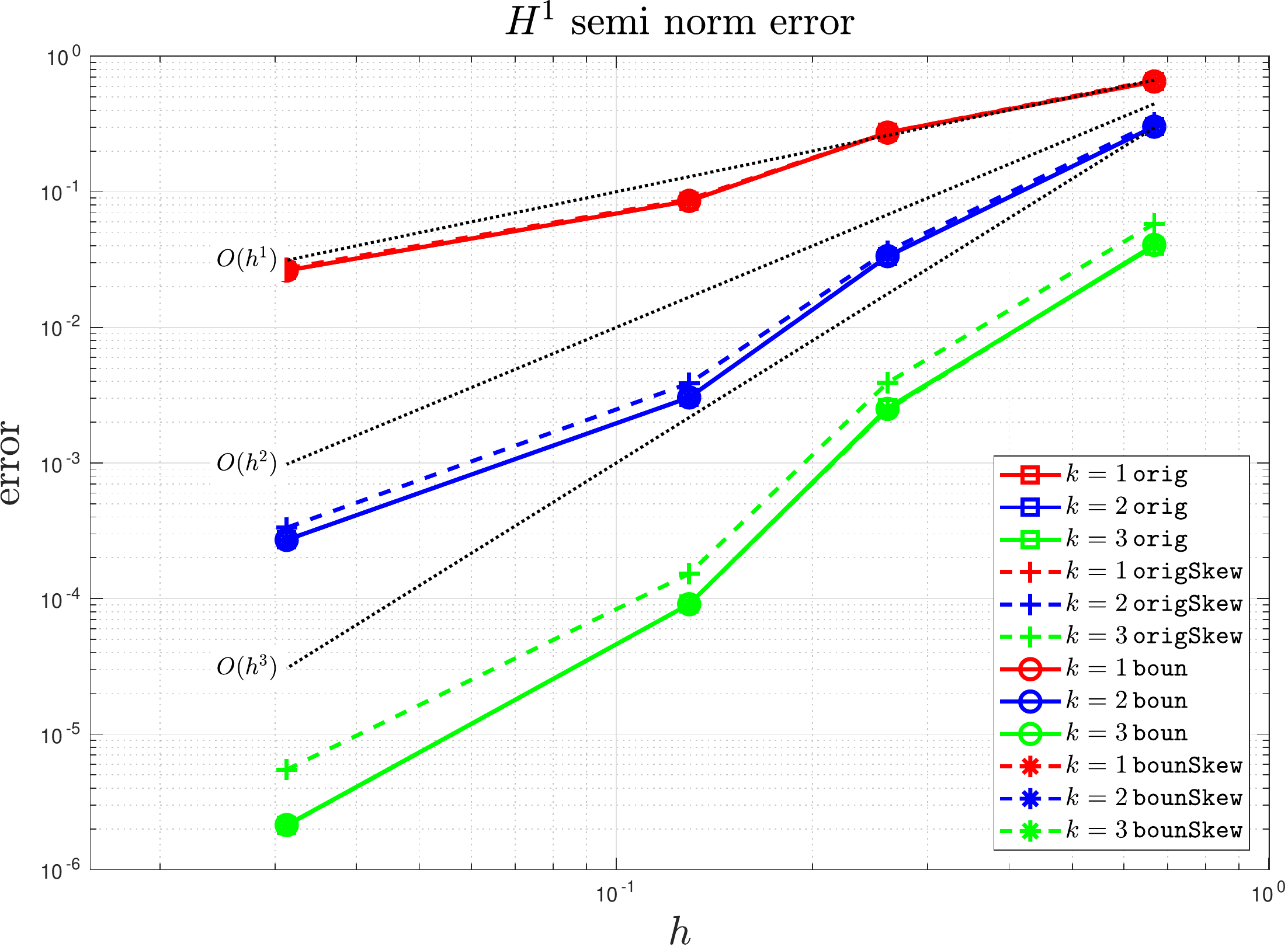} &
\includegraphics[width=0.47\textwidth]{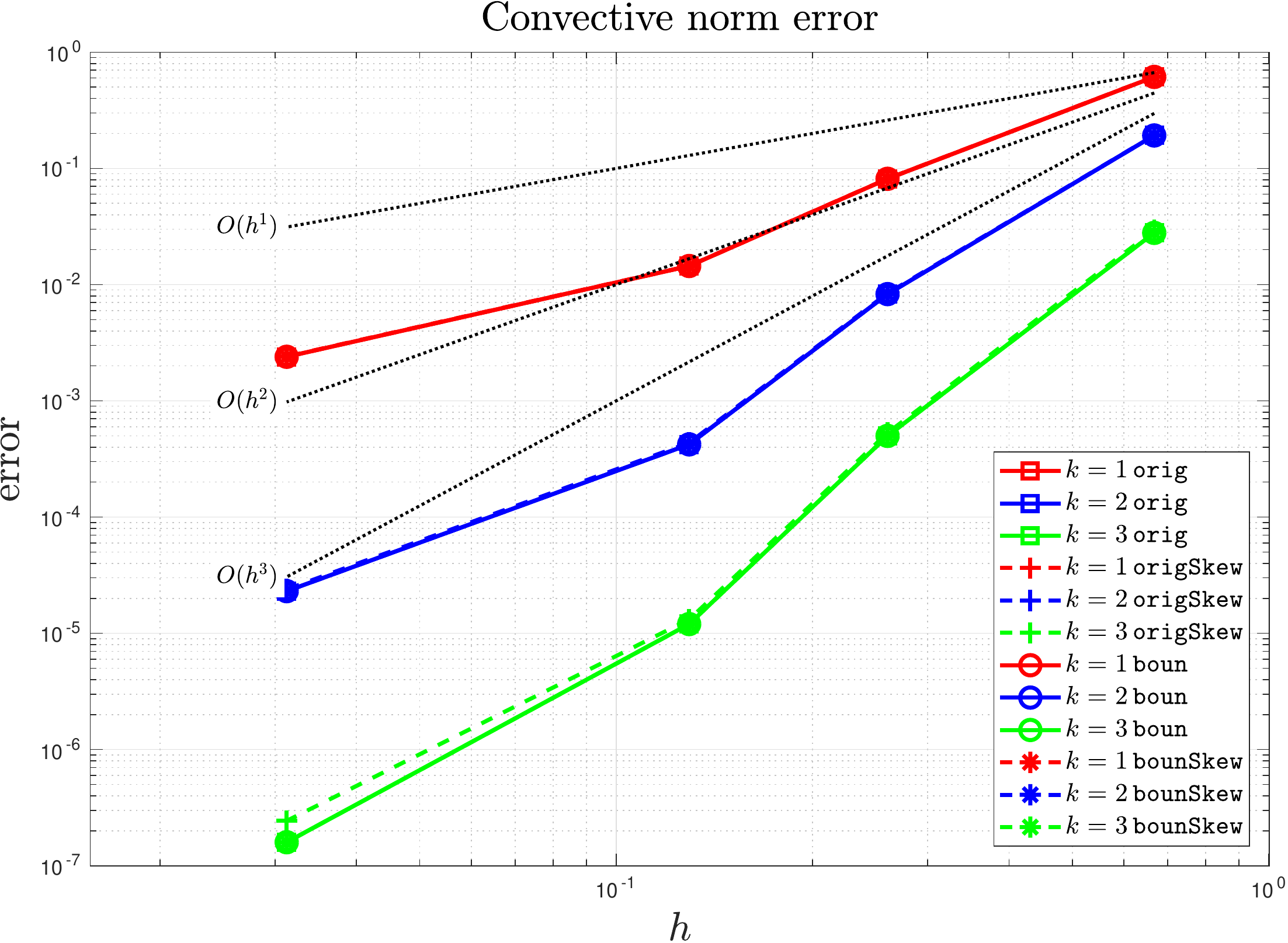} \\
\end{tabular}
\end{center}
\caption{Convergence lines for the case $\epsilon=1e-03$}
\label{fig:convKappaM3Bdiff}
\end{figure}

\begin{figure}[!htb]
\begin{center}
\begin{tabular}{cc}
\multicolumn{2}{c}{\texttt{quad}} \\
\includegraphics[width=0.47\textwidth]{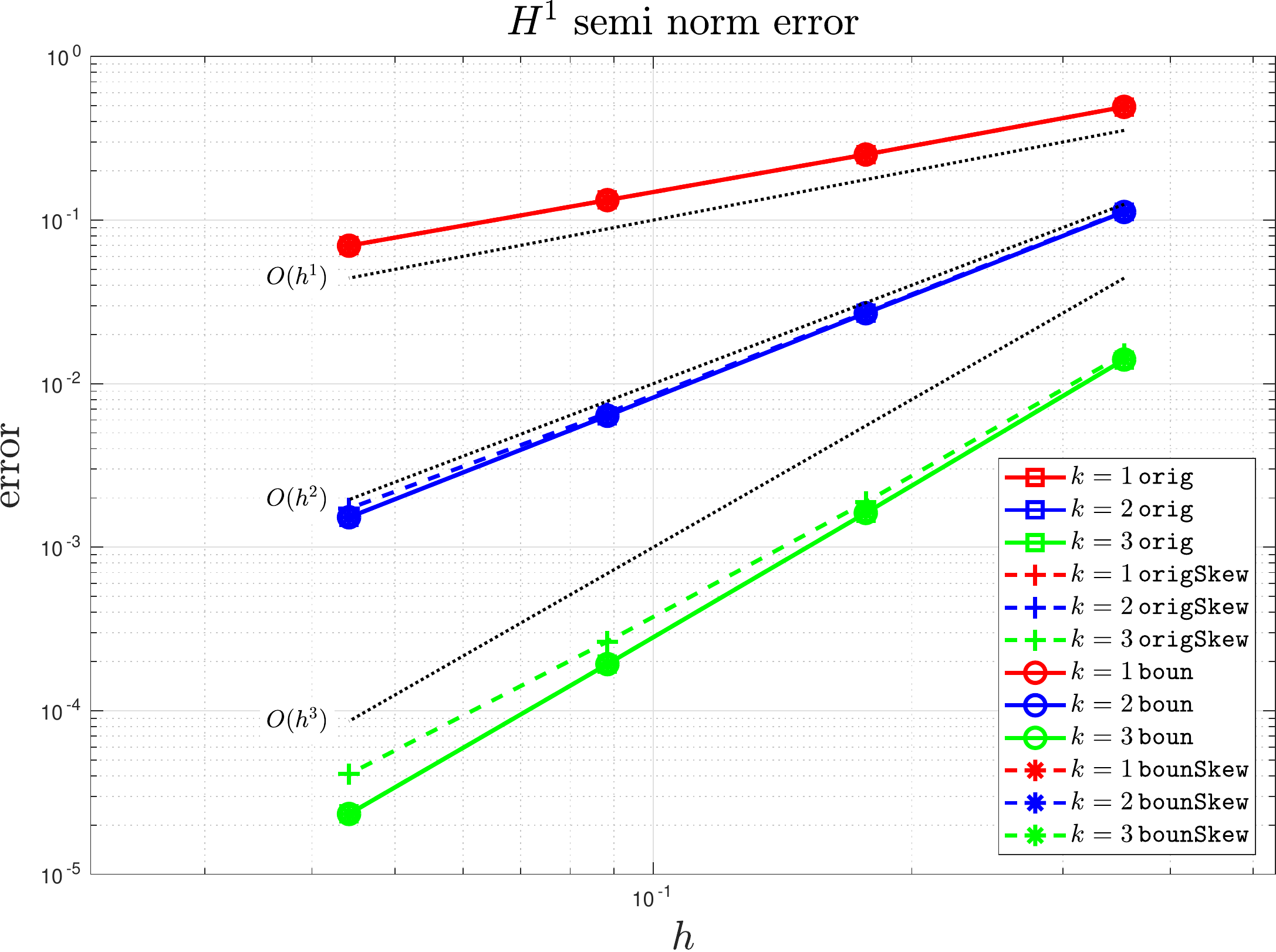} &
\includegraphics[width=0.47\textwidth]{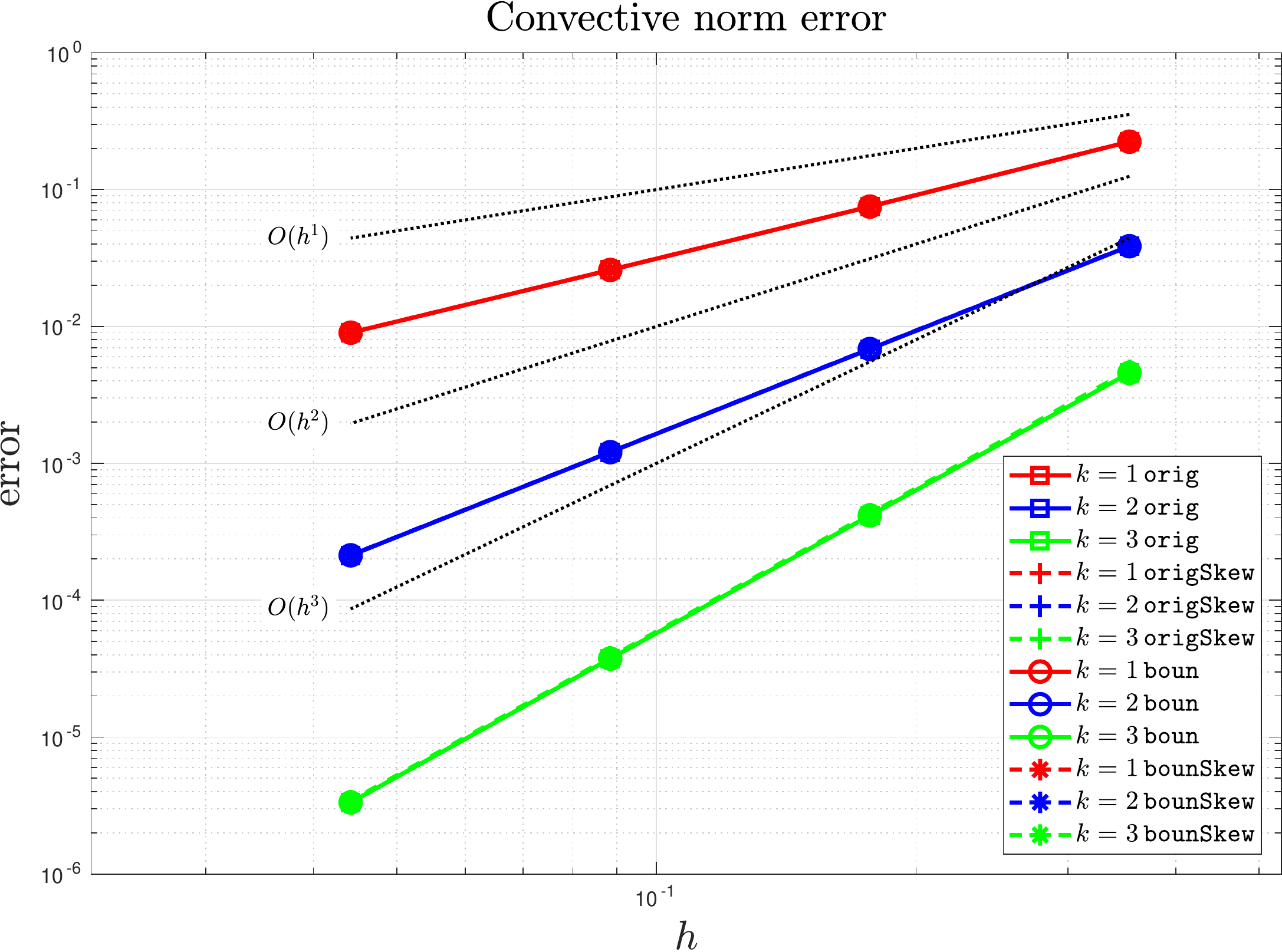} \\
\multicolumn{2}{c}{\texttt{tria}} \\
\includegraphics[width=0.47\textwidth]{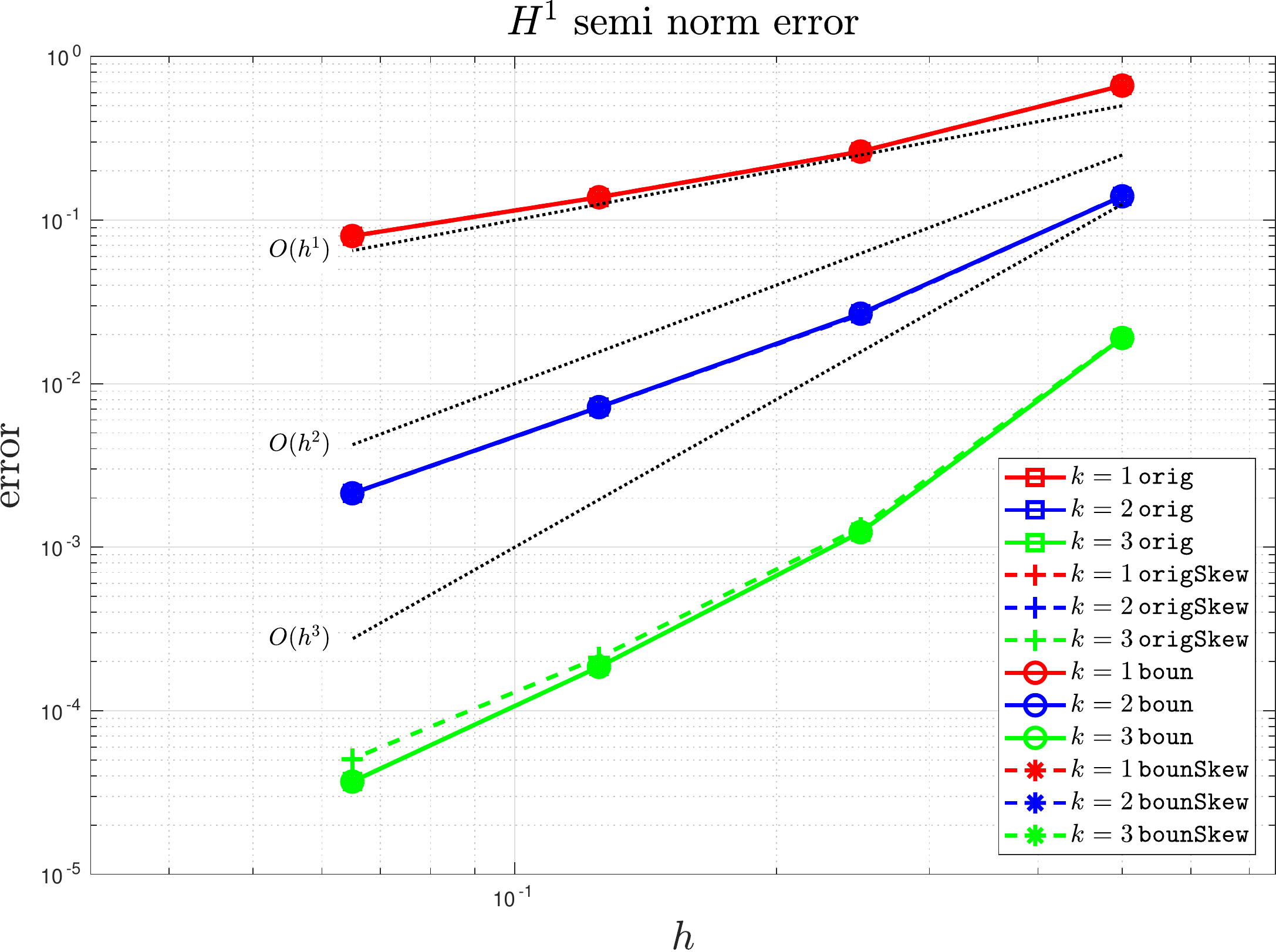} &
\includegraphics[width=0.47\textwidth]{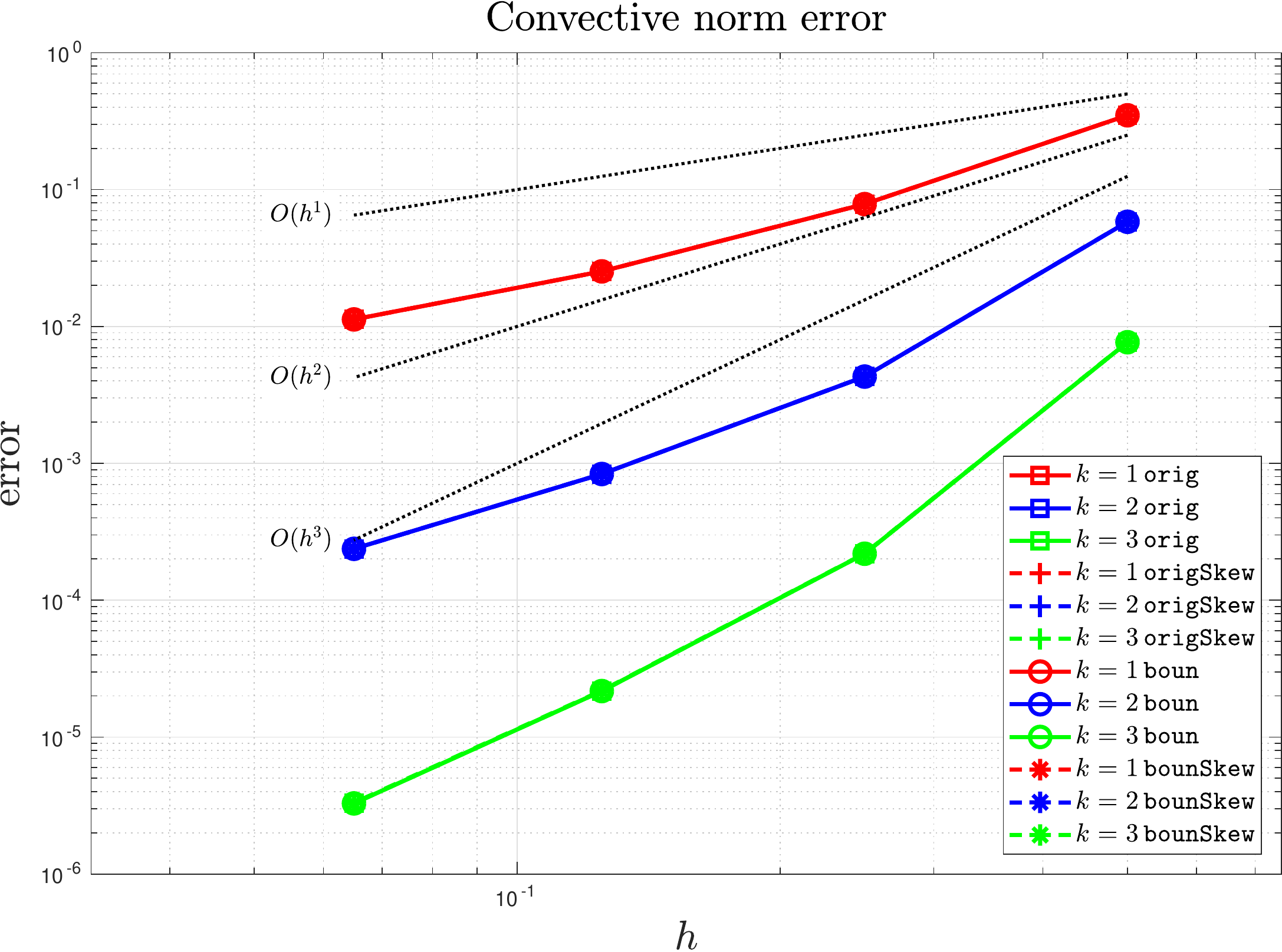} \\
\multicolumn{2}{c}{\texttt{voro}} \\
\includegraphics[width=0.47\textwidth]{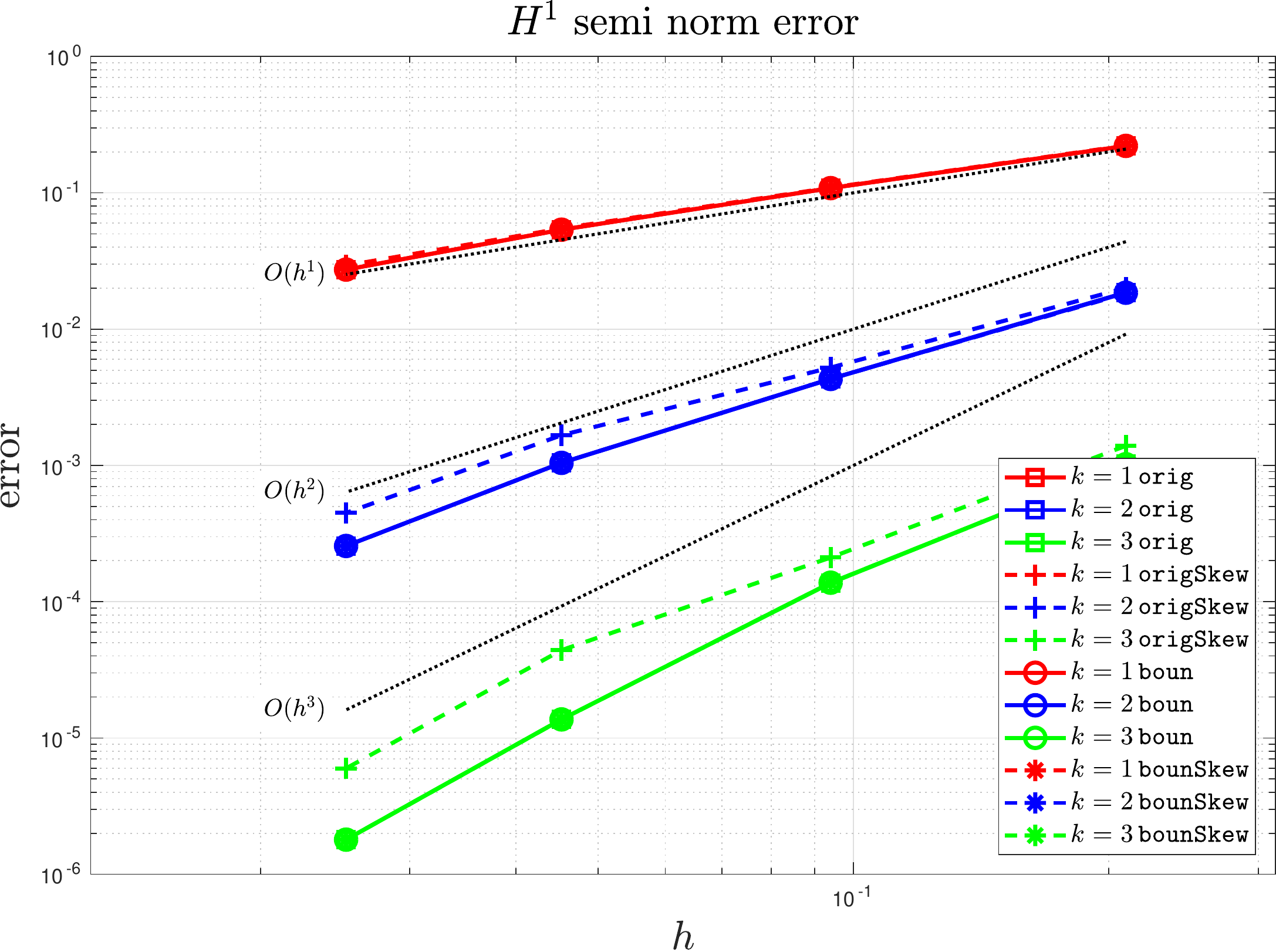} &
\includegraphics[width=0.47\textwidth]{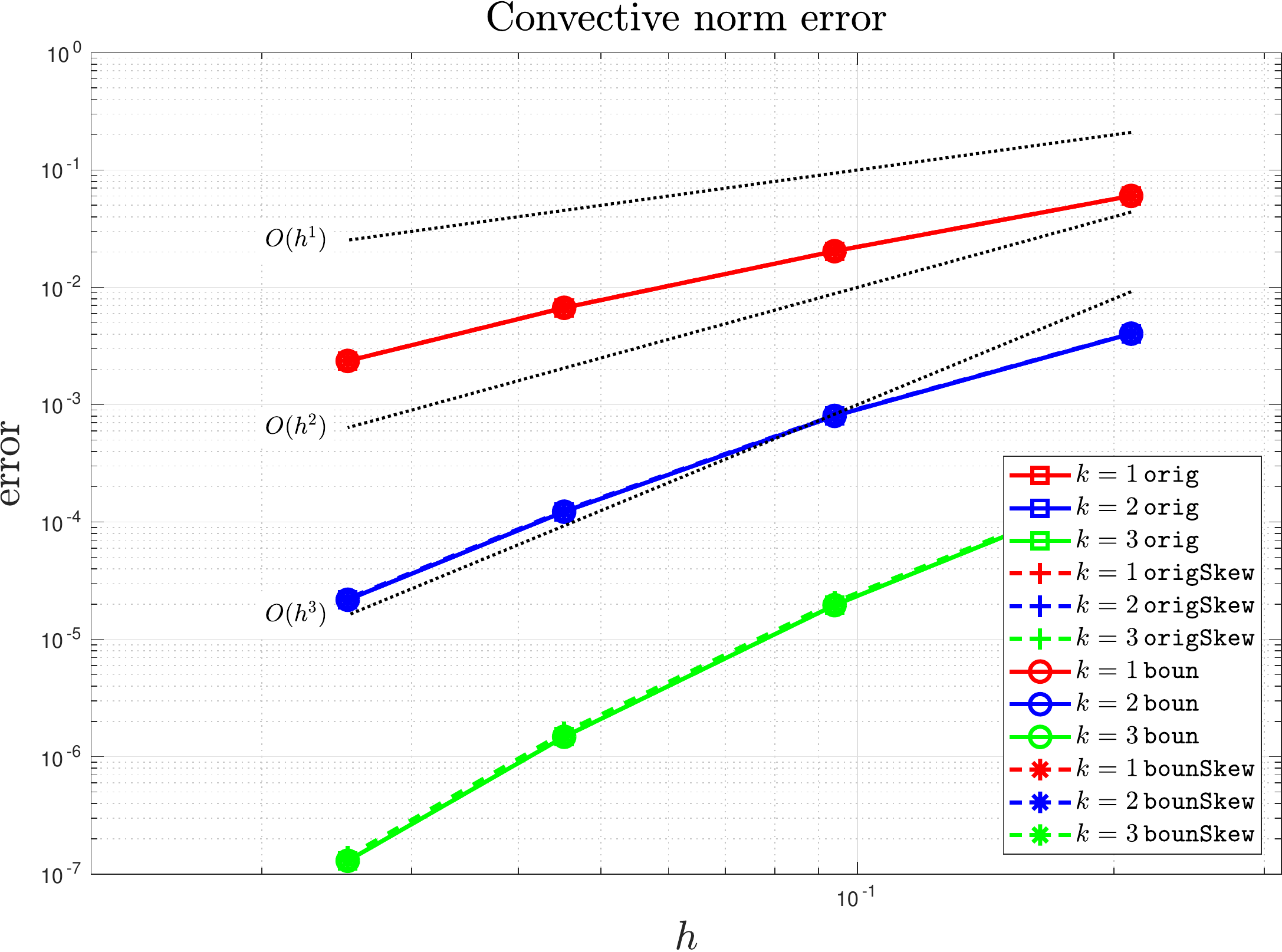} \\
\multicolumn{2}{c}{\texttt{rand}} \\
\includegraphics[width=0.47\textwidth]{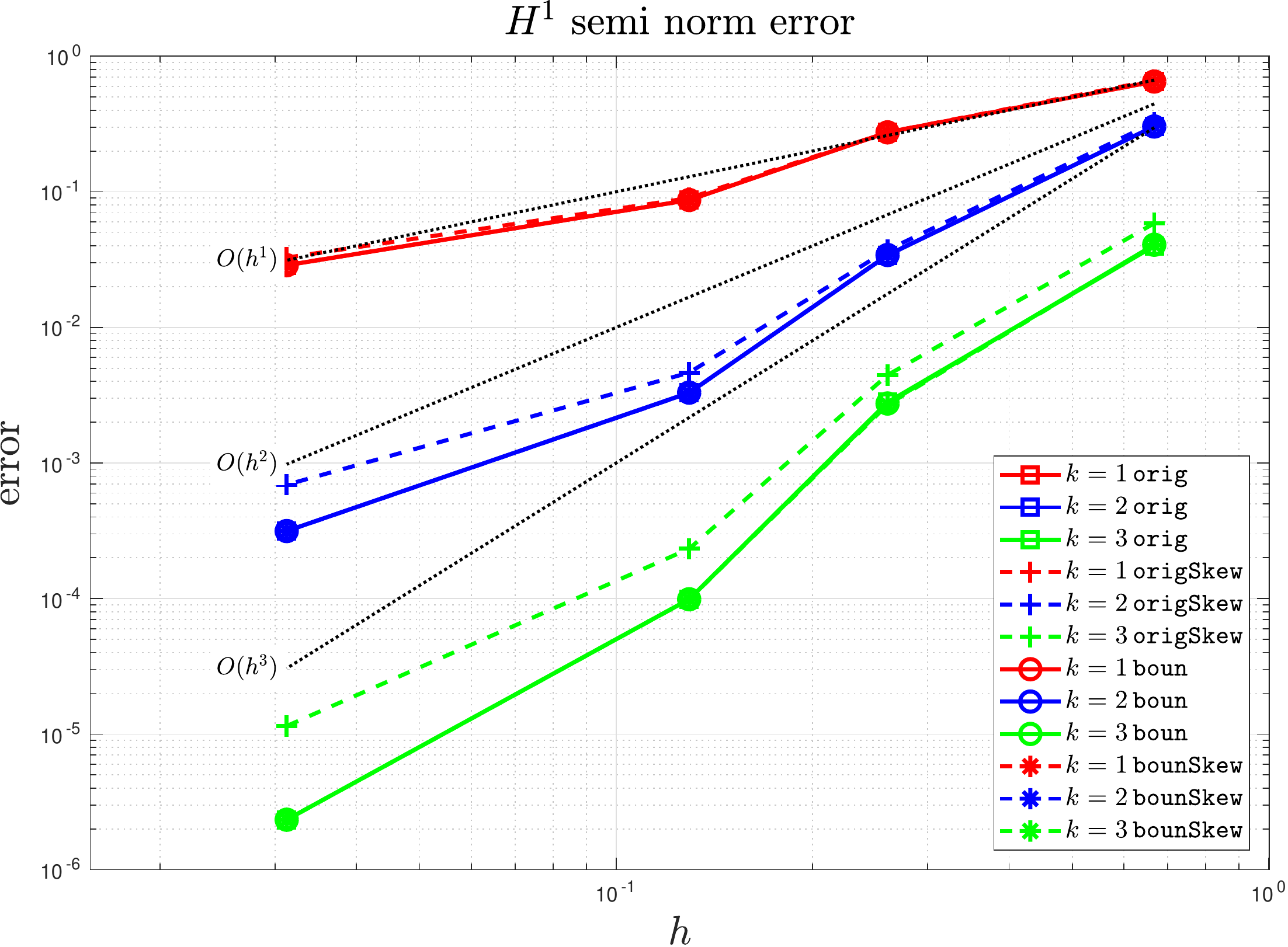} &
\includegraphics[width=0.47\textwidth]{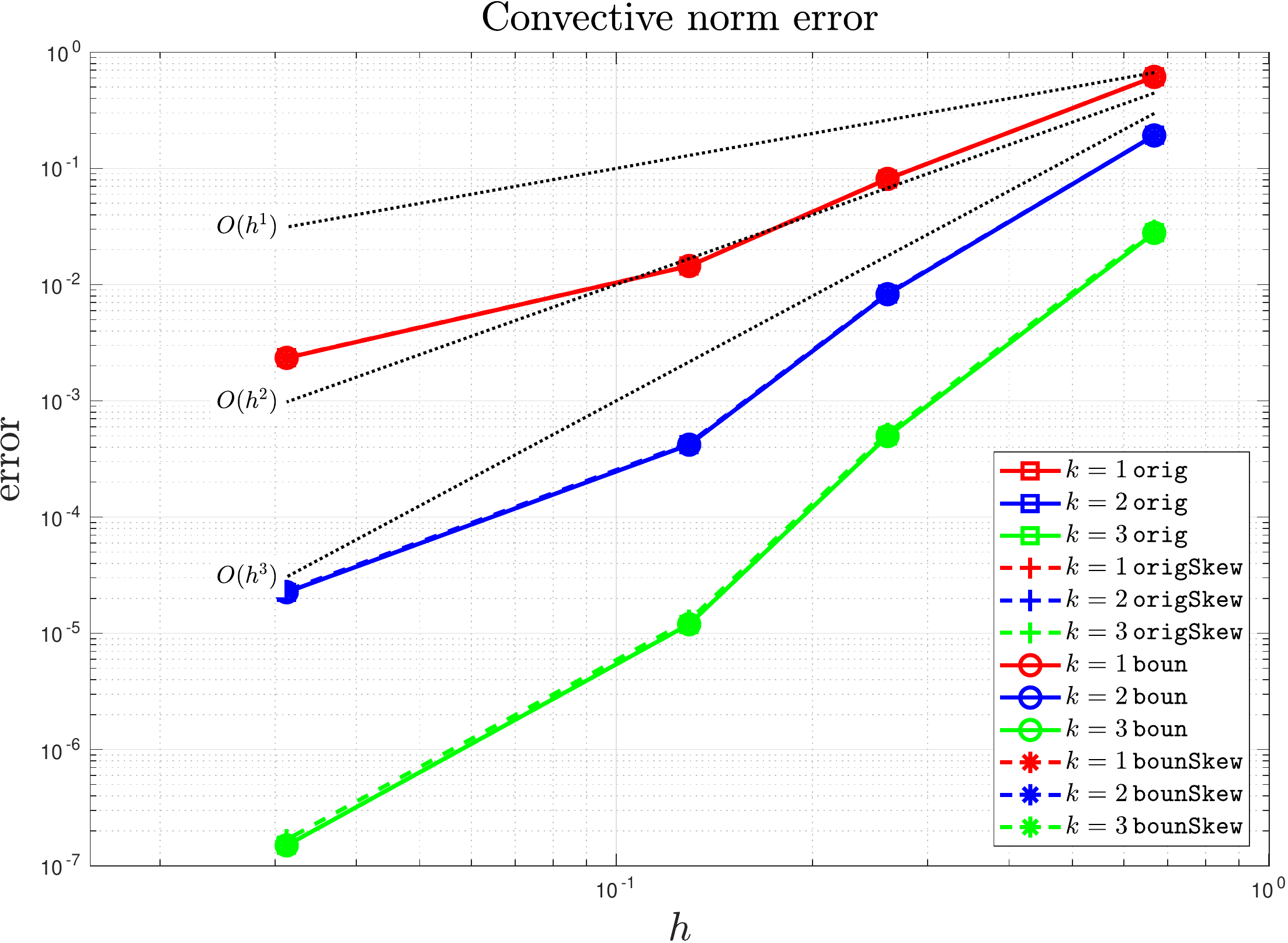} \\
\end{tabular}
\end{center}
\caption{Convergence lines for the case $\epsilon=1e-06$}
\label{fig:convKappaM6Bdiff}
\end{figure}

We conclude the section by noticing that  $b_{\partial,h}^E(\cdot,\cdot)$ 
and $b_{o,h}^E(\cdot,\cdot)$  coincide for a costant convection term $\bb$ (the proof can be easily performed by a direct computation). 
To have a numerical evidence about this fact,
we have considered a problem with constant vector field $\bb$ and compared the stiffness matrices provided by $b_{o,h}^E(\cdot,\cdot)$ and $b_{\partial,h}^E(\cdot,\cdot)$, respectively. For every mesh and every approximation degree, we have found that they always differ up to machine precision. Here we show the data (norms of the difference between the stifness matrices) only for the finest \texttt{voro} mesh and for $k=1,2$, see Table~\ref{tab:diffMat}.

\begin{table}[!htb]
\begin{center}
\begin{tabular}{|c|c|c|}
\hline
\multicolumn{1}{|c|}{$k$} &\multicolumn{2}{c|}{\texttt{voro}} \\
\hline
1 &1.1221e-15 &1.7536e-17 \\ 
2 &9.2636e-16 &2.3161e-17 \\ 
\hline
\end{tabular}
\caption{Diffference between the stiffness matrices considering $b_{o,h}^E(\cdot,\cdot)$ or $b_{\partial,h}^E(\cdot,\cdot)$.}
\label{tab:diffMat}
\end{center}
\end{table}